\documentclass{amsart}

\usepackage[latin1]{inputenc}%
\usepackage{amssymb,amsfonts,epsfig,amsmath,graphics,graphicx}%
\usepackage{amsthm}%
\usepackage[all]{xy}%
\usepackage{verbatim}%
\usepackage{enumitem}
\usepackage{bm}

\title {Compactification and trees of spheres covers}
\date{}

\newtheorem{theorem} {Theorem}[section]
\newtheorem{proposition}[theorem]{Proposition}
\newtheorem{lemma}[theorem]{Lemma}
\newtheorem{corollary}[theorem]{Corollary}
\newtheorem{definition}[theorem]{Definition}

\theoremstyle{plain}
\newtheorem{theoremint} {Theorem}
\newtheorem*{definitionint}{Definition}

 \newtheorem*{claimint}{Claim}
 \newtheorem*{corollaryint}{Corollary}
 
 \newtheorem{claim}{Claim}

\theoremstyle{remark}
\newtheorem{remark}[theorem]{Remark}

\renewenvironment{proof}{\noindent{\bf Proof. }}{\hfill{$\square$} \vskip.3cm}

\def\cal{\mathcal}

\def\C{{\mathbb C}}

\def\F{{\mathcal F}}
\def\I{{\mathcal I}}

\def\N{{\mathbb N}}
\def\P{{\mathbb P}}
\def\Pt{{\mathcal P}}

\def\S{{\mathbb S}}

\def\T{{\mathcal T}}

\def\Aut{{\rm Aut}}
\def\Rat{{\rm Rat}}
\def\rat{{\rm rat}}

\def\RevTC{{\widehat{\bf Rat}}}
\def\DRevTC{{{\bf Dyn}}}
\def\drevTC{{{\bf dyn}}}
\def\revTC{{\widehat{\bf rat}}}
\def\revT{{\bf rat}}
\def\RevT{{\bf Rat}}

\def\ts{{tree of spheres }}

\def\CST{\bf{\bf{\mathfrak T}}}
\def\M{{\rm Mod}}
\def\MT{{\bf Mod}}
\def\CM{{\overline {\rm Mod}}}
\def\CMT{{\widehat{\bf Mod}}}
\def\MI{{\bf{\bf{ I}}}}

\def\Top{{\rm {\mathfrak B}   }}
\def\TopT{{\bf {\mathfrak B} }}

\def\OubQ{{\bf {\bf \Pi} }}
\def\Oub{{\Pi}}

 \def\epsilon{{\varepsilon}}

\def\card{{\rm card}}
\def\deg{{\rm deg}}

\def\VConf{{\rm VConf}}
\def\VM{{\rm VM}}

\usepackage{hyperref}

\hypersetup{
backref=true, 
pagebackref=true,
hyperindex=true, 
colorlinks=true, 
breaklinks=true, 
urlcolor= blue, 
linkcolor= blue, 
}

  \author[Matthieu Arfeux]{Matthieu Arfeux}                                
    \address{Stony Brook University,Stony Brook, NY 11794}                                    
    \email{matthieu.arfeux@stonybrook.edu}                                      
   \begin{document}

\maketitle

\begin{abstract}
 The space of dynamically marked rational maps can be identified with a subspace of the space of covers between trees of spheres on which there is a notion of convergence that makes it sequentially compact. In this paper we describe a topology on the quotient of this space under the natural action of its group of isomorphisms. This topology is proved to be consistent with this notion of convergence.
\end{abstract}

\let\thefootnote\relax


\section{Introduction}


We denote by $\S:= \P^1 (\C)$ the Riemann sphere.
For $d\geq 1$, we denote by $\Rat_d$ the space of rational maps $f:{\mathbb S}\to {\mathbb S}$ of degree $d\geq 1$ endowed with the topology of uniform convergence. In particular, $\Aut(\S):=\Rat_1$ is the group of Moebius transformations and acts on $\Rat_d$ by conjugacy :
\[\Aut(\S)\times \Rat_d\ni (\phi,f)\mapsto \phi\circ f\circ \phi^{-1}\in \Rat_d.\]  
We are interested in the quotient $\rat_d$ of $\Rat_d$ by this action. The space $\rat_d$ is not compact for any $d$.
In this paper we study spaces of rational maps marked by portraits and a compactification of the associated moduli space defined below.

Let $X$ be a finite set with at least 3 elements. A {\it sphere marked by $X$} is an injection $x:X\to \mathbb S$.
A {\it portrait} ${\bf F}$ of degree $d\geq 2$ is a pair $(F,\deg)$ where 
\begin{itemize}
\item $F:Y\to Z$ is a map between two finite sets $Y$ and $Z$ and
\item $\deg:Y\to \{1,2,\ldots,d\}$ is a function that satisfies
\[\sum_{a\in Y}\bigl(\deg(a) -1\bigr) = 2d-2\quad\text{and}\quad \sum_{a\in F^{-1}(b)} \deg(a) = d\quad\text{ for all } b\in Z.\] 
\end{itemize}
Typically, $Z\subset \S$ is a finite set, $F:Y\to Z$ is the restriction of a rational map $f:\S\to \S$ to $Y:=f^{-1}(Z)$ of degree $d$ and $\deg(a)$ is the local degree of $f$ at $a$. In this case, the Riemann-Hurwitz formula and the conditions on the function $\deg$ implies that $Z$  contains the set of critical values of $f$ so that $f:\S-Y\to \S-Z$ is an unbranched cover. 

\begin{definitionint}[Marked rational map]
A rational map marked by a degree $d$ portrait ${\bf F}$ is a triple $(f,y,z)$ where
\begin{itemize}
\item $f\in \Rat_d$
\item $y:Y\to \S$ and $z:Z\to \S$ are marked spheres, 
\item $f\circ y = z\circ F$ on $Y$ and 
\item $\deg_{y(a)}f = \deg(a)$ for $a\in Y$. 
\end{itemize}
\end{definitionint}

If $(f,y,z)$ is marked by ${\bf F}$, we have the following commutative diagram : 

\centerline{
$\xymatrix{
 Y \ar[r]^{y}    \ar[d] _{{ F}} &\S  \ar[d]^{f} \\
      Z \ar[r]_{z}  &\S.
  }$}
  Note that in particular, from the definition of ${\bf F}$ it follows that $y(Y)=f^{-1}(z(Z))$.

\begin{definitionint}[Dynamically marked rational map]
A rational map dynamically marked by  $({\bf F},X)$ is a triple $(f,y,z)$ which is a rational map marked by ${\bf F}$ such that  $X\subseteq Y\cap Z$ and $y|_X=z|_X$. 
\end{definitionint}

We denote by $\Rat_{\bf F}$ the set of rational maps marked by ${\bf F}$ and $\Rat_{{\bf F},X}$ the set of rational maps dynamically marked by $({\bf F},X)$. 

The group $\Aut(\S)$ acts on $\Rat_{\bf F}$ by pre-composition and post-composition: a pair of Moebius transformations $(\phi,\psi)\in \Aut(\S)\times \Aut(\S)$ maps the marked rational map $(f,y,z)\in \Rat_{\bf F}$ to 
\[(\phi\circ f\circ \psi^{-1},\psi\circ y, \phi\circ z)\in \Rat_{\bf F}\]
as in the following diagram:

\centerline{
$\xymatrix{
 Y \ar[r]^{y}    \ar[d] _{{ F}} &\S  \ar[d]^{f}\ar[r]^\psi& \S\ar[d] ^{\phi\circ f \circ \psi^{-1}}\\
      Z \ar[r]_{z}  &\S\ar[r]_\phi &\S.
  }$}
  
We denote by $\rat_{\bf F}$ the quotient of $\Rat_{\bf F}$ by the action of $\Aut(\S)\times \Aut(\S)$. 

Likewise, the group $\Aut(\S)$ acts on $\Rat_{{\bf F},X}$ by conjugacy: a Moebius transformation $\phi\in \Aut(\S)$ maps the dynamically marked rational map $(f,y,z)\in \Rat_{{\bf F},X}$ to 
\[(\phi\circ f \circ \phi^{-1},\phi\circ y,\phi\circ z)\in \Rat_{{\bf F},X}.\]
We denote by $\rat_{{\bf F},X}$ the quotient of $\Rat_{{\bf F},X}$ by the action of $\Aut(\S)$.  

For any finite set $E$ with at least three elements, we define the {\it moduli space} $\M_E$ to be the space of spheres marked by $E$ modulo post-composition by Moebius transformations. As $\card X\geq 3$ and if $(f,y,z)\in \Rat_{\bf F}$, then $f$ is determined by the pair $(y,z)$.  Indeed, a rational map is totally determined if we know the preimages, with multiplicities, of any triple of points. Thus $\rat_{\textbf{F},X}$ naturally injects into the product of the moduli space of spheres marked by $Y$ and by $Z$,
$$\rat_{\textbf{F},X}\to\M_Y\times\M_Z.$$

In fact, one can prove that given a portrait ${\bf F}$, the natural projection $$\rat_{\textbf{F}}\to\M_Y$$ is also injective, i.e. $[y]\in\M_Y$ characterizes $[(f,y,z)]\in\rat_{{\bf F}}$.


\medskip
\noindent{\textbf{Main goal.}}

In \cite{FT} R. Funahashi and M. Taniguchi define a compactification $\CM_E$ of $\M_E$ by adding equivalence classes of noded spheres. This compactification can be identified with special cases of compactifications introduced by Deligne and Mumford in \cite{DM} and by Bers in \cite{Bers}. 
In Section \ref{chap5}, we rephrase \cite{FT} with the vocabulary of trees of spheres from \cite{A1}. We will denote this compactification by $\CM_E$. 
 We introduce the set $\CMT_E$ of trees of spheres marked by $E$ (cf Figure \ref{figexample}) modulo a certain notion of isomorphism between trees of spheres. 
Considering $\rm Quad_E$, the set of quadruples of distinct elements of $E$, we prove the following:

\begin{theoremint}\label{prop519}
By taking cross ratios, the image of the embedding 
$$\Top:{\CMT}_E\to  \S ^{ \rm Quad_E}$$
 is equal to the \cite{FT} closure of $\M_E$.
\end{theoremint}

\begin{corollaryint}
The map $\Top$ endows ${\CMT}_E$ with a topology that makes it compact.
\end{corollaryint}

We define a convergence notion on ${\CMT}_E$ as in \cite{A1}, and prove that it is compatible with this topology on $\CM_X$.
Furthermore, we extends the results of \cite{FT} to give a combinatorial characterization of the elements of $\Top({\CMT}_E)$ (cf Theorem \ref{bijpartition}). This characterization will be particularly useful later in the article.


In section \ref{chap6}, we compactify $\rat_{{\bf F}}$ of as a subspace of $\M_Y$ by taking its closure in $\CM_Y$. To do this,
we introduce $\revTC_{{\bf F}}$ to be the set $\RevTC_{{\bf F}}$ of covers between trees of spheres modulo some notion of isomorphism.  The space $\revT_{{\bf F}}$ is the subspace consisting in classes of covers between trees with only one internal vertex and it can be identified with $\rat_{{\bf F}}$.  

 \begin{theoremint} \label{compseqrev}
 The natural projection map 
 $\MI: \revTC_{{\bf F}}\to {\MT}_Y$ is injective and the topological space $\revTC_{{\bf F}}$ is compact as the closure of $\revT_{{\bf F}}$, 
$$\text{ie:}\quad\MI(\revTC_{{\bf F}})=\overline{{\MI}(\revT_{{\bf F}})}.$$
 \end{theoremint}

 We define a notion of convergence on $\RevTC_{{\bf F}}$ as in \cite{A1}
 and we prove that it is compatible with this topology.
 
%


In Section \ref{chap7}, we identify $\rat_{{\bf F},X}$ with a subset of $\rat_{{\bf F}}$. Hence, it is also endowed with a topology. We introduce $\drevTC_{{\bf F},X}$ the set of dynamical systems between trees of spheres modulo a certain notion of isomorphism and prove that it can also be identified with a subset of $\revTC_{{\bf F}}$. The elements of $\overline{\rat_{{\bf F},X}}$ can be identified with isomorphism classes of dynamical systems between trees of spheres. 
We prove the following result.

  \begin{theoremint}\label{ccompact}  The space $\drevTC_{{\bf F},X}$ is compact.
\end{theoremint} 

We prove in particular that this topology is compatible with the dynamical convergence defined in \cite{A1}. We show that the discussions in \cite{A2} prove the following inclusions:
$$\overline{\rat_{{\bf F},X}}\subsetneq\drevTC_{{\bf F},X}\subsetneq \revTC_{{\bf F}}. $$

%

\medskip
\noindent{\it{Important note:}} Mathematically, this paper has to be considered before \cite{A1} as it is totally independent from it whereas \cite{A1} cites this paper. However, \cite{A1} is a better introduction to the motivations behind the study of such questions and the author spends more time to describe the tools used here. 


\medskip
\noindent{\textbf{Acknowledgments.}}

 I would want to thanks my advisor Xavier Buff for all the time he spent to teach me how to write and make clear my ideas. I also want to thank Jan Kiwi and Laura De Marco for their helpful comments.


\section{Isomorphism classes of trees of spheres}\label{chap5}

\subsection{Introduction}
In this subsection we recall notions and notations introduced in \cite{A1}.

 Let $X$ be a finite set with at least 3 elements. A (projective) tree of spheres ${\cal T}$ marked by $X$ (or $\T^X$, cf example on Figure \ref{figexample}) is the following data : 
\begin{itemize}
\item a combinatorial tree $T$ (or $T^X$) marked by $X$, disjoint union of its set of vertices $V$ and its set of edges $E$, whose set of leaves (external vertices) is $X$ and such that every internal vertex has at least valence $3$ (stability), and
\item for each internal vertex $v$ of $T$, an injection $i_v:E_v\to {\S}_v$ of the set $E_v$ of edges adjacent to $v$ into a projective sphere ${\S}_v$.
\end{itemize}
 We denote by $\{v,v'\}$ the edge between two vertices $v$ and $v'$ if it exists.
 We use the notation $X_v := i_v(E_v)$ and define the map $a_v:X\to {\S}_v$ such that $a_v(x) := i_v(e)$ if $x$ and $e$ lie in the same connected component of $T-\{v\}$. We say that $i_v(e)$ is the attaching point of the edge $e$ on $v$ or $\S_v$.
 
 For $\star\in T\setminus\{v\}$ we denote by $B_v(\star)$ the connected component of $T\setminus\{v\}$ containing $\star$. It is called the {\it branch} at $v$ containing $\star$. For any vertex $v'$ in a branch $B_v(e)$ where $e$ is an edge adjacent to $v$. We will also make use of the notation $i_v(v')$ for the attaching point $i_v(e)$.
 
  We denote by $[v,v']$ the arc (i.e. injective path) between $v$ and $v'$. We will also denote it by $[v_1,v_2,\ldots v_k]$ if $v=v_1$, $v'=v_k$ and the vertices $v_i$ are adjacent to the $v_{i+1}$.

 \begin{figure} \centerline{\includegraphics[width=9cm]{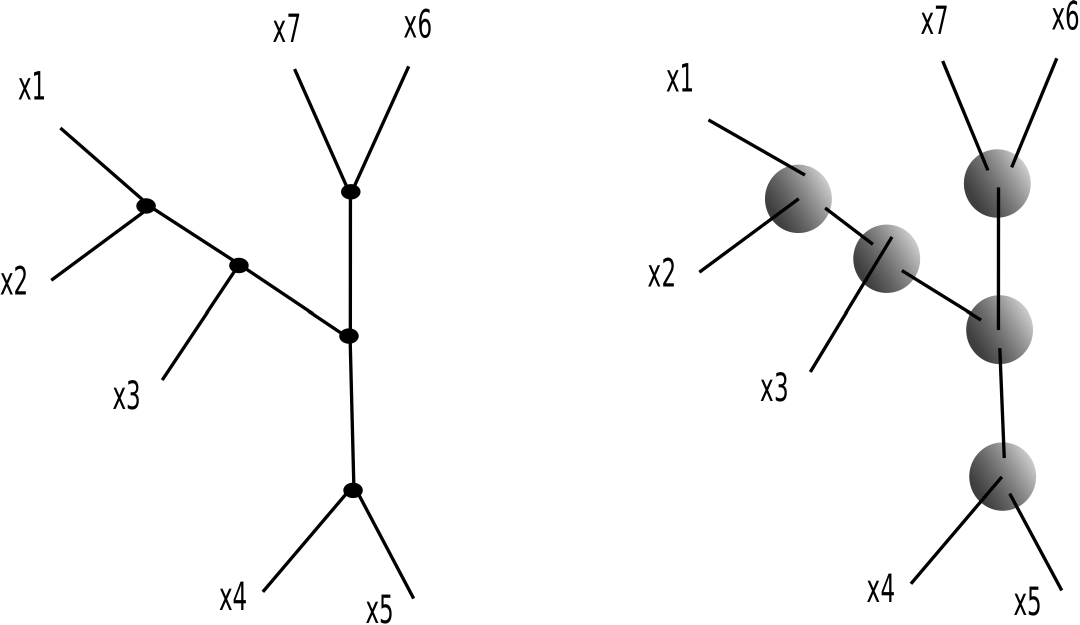}}  \caption{On the left, a combinatorial tree $T^X$ for ${X=\{x1\ldots x7\}}$. On the right, the representation of a tree of spheres whose combinatorial tree is $T^X$.}\label{figexample}\end{figure}
 
We will later identify trees with only one internal vertex with marked spheres defined below.

\begin{definition}
A sphere marked by $X$ is an injection $$x:X\to \mathbb S.$$
\end{definition}


\subsection{Isomorphism of combinatorial trees and partitions}

\begin{definition}[Tree maps]
A (combinatorial) tree map is a map between two combinatorial trees that maps internal vertices to internal vertices, external ones to external ones, edges to edges and two adjacent vertices to adjacent ones.
\end{definition}

\begin{definition}[Isomorphism of marked trees]
An isomorphism between two combinatorial trees marked by $X$ is a tree map which is bijective and restricts to the identity on $X$.
\end{definition}

Define $P_v:= \{  B_{v}(e) \cap X ~|~ e\in E_v\}$.
By convention, a partition does not contain the element $\emptyset$.
%
%
For all $v$ internal vertex, $P_v$ is a partition of $X$.
%

Let $\psi$ be the map between the set of trees marked by $X$ and the set of collections of partitions of $X$ that maps $T$ to $$\psi(T)=\{ P_v\subset X| v\text{ internal vertex of }T^X\}.$$ 

The goal of this section is to give a characterization of the image of this map and of the isomorphism classes of combinatorial trees marked by $X$.

\begin{definition}[Admissible collection of partitions]\label{defadmpart}
A collection ${\Pt}$ of partitions is admissible if it satisfies the following properties : 
\begin{enumerate}
\item\label{mt1} every partition ${ P}\in {\Pt}$ contains at least three distinct elements, 
\item\label{mt2} for all partition ${P}\in {\cal  P}$ and all subset $B\in { P}$, either there exists a partition ${ P}'\in {\cal  P}$ containing $X-B$, or $B=\{x\}$ with $x\in X$,
\item\label{mt3} if ${ P}_1\in {\cal  P}$ and ${ P}_2\in {\cal  P}$ are two distinct partitions, then ${ P}_1\cap { P}_2=\emptyset$. 
\end{enumerate}
\end{definition}

\begin{remark}\label{iasdmissparti}
If $T$ is a combinatorial tree, then $\psi(T)$ is clearly a collection of admissible partitions (a detailed proof can be found in \cite{A}).
\end{remark}

We prove the following theorem:

\begin{theorem}\label{bijpartition}
Two combinatorial trees $T$ and $T'$ are isomorphic if and only if ${\psi(T)=\psi(T')}$.
If $T$ is a combinatorial tree, then $\psi(T)$ is a collection of admissible partitions, and every admissible collection of partition is the image of a (stable) tree.
\end{theorem}

\begin{corollary}\label{propsurj}
The map $\psi$ induces a bijection between the set of isomorphism classes of trees and the set of admissible collections of partitions.
\end{corollary}

We finish this section by the proof of Theorem \ref{bijpartition}.

Take any admissible collection of partitions $\cal P$.  Define a set of vertices ${V_T=\Pt\cup X}$. Define a set of edges $E_T$ as the set of $\{P_1,P_2\}$ for all $P_1\in \cal P$ and $P_2\in \cal P$ such that we have $B_i\in P_1$ and $B_2\in P_2$ satisfying $B_1\cup B_2= X$ with $B_1\cap B_2=\emptyset$ and the $\{  P_0,x \}$ satisfying $P_0\in \cal P$ and $\{   x\} \in P_0$. We define the graph $T_{\cal P}$ to be $V_T\sqcup E_T$.

\begin{lemma} \label{affirmation}
Let $x\in X$.
Every vertex $v\in V_T{\setminus}\{x\}$ can be connected to the vertex $x$ by a unique arc. Moreover, if the first edge of this arc is $\{v,P\}$, then $x\in P$. 
\end{lemma}

\begin{proof} 
Define $v_1:=v$. We are looking for an arc $[v_1,v_2,v_3,\ldots,v_k,x]$.  

 If $v_1\in X$ then the second property assures the existence of a vertex $v_2$ such that $\{v_1\}$ lies in the partition. Then we are in the case $v_1\notin X$.
We find the $v_i$ recursively. 

Recurrence hypothesis: there exist $v_2,\ldots,v_i$ such that $[v_1,v_2, \ldots,v_i]$ is an arc and the element $B_i$ of $v_i$ containing $x$ is included in the one of $v_{i-1}$ containing $x$. 

Suppose that it is true for some $i\in \N$.
Let $B_i$ be this element. If $B_i=\{x\}$ then by construction $\{v_i,\{x\} \}\in E_T$ and  $[v_1,v_2, \ldots,v_i,\{x\}]$ is the desired arc.
 If not, we find $v_{i+1}$ containing $X{\setminus} B_i\in V_T$. Thus $\{v_i,v_{i+1} \}\in E_T$. 
If $B_{i+1}$ is the subset of $v_{i+1}$ containing $x$ then $B_{i+1}$ and $ X{\setminus} B_i$ are two elements of the partition $v_{i+1}$ so we have $B_{i+1}\subset B_i$ as desired. 
The property is true for $i+1$.

This construction stops because the inclusions of $B_i$ are strict. In addition we always have $x\in B_i$. It follows that if $v_k$ is the last vertex of the constructed arc then $v_k=x$. 

This arc is unique because the hypothesis $B_{i+1}\subset B_i$ is necessary and induces the unicity for the choice of the vertices choices at every step.
The end of the lemma follows from the construction.
\end{proof}

\begin{corollary}
The graph $T_{\cal P}$ is a combinatorial tree marked by $X$ and $\psi(T_{\cal P})=\cal P$.
\end{corollary}

\begin{proof} 
For $T_\cal P$, the connectivity and the non existence of cycle is clear from Lemma \ref{affirmation}. The stability is direct from the first property of Definition \ref{defadmpart}, and by construction the leaves of our tree are the elements of $X$. Hence $T_{\cal P}$ is a combinatorial tree marked by $X$.

Now we prove that $\psi(T_{\cal P})=\cal P$.
Let $v_1\in V_T$. Denote by ${ P}=\{p_1,\ldots,p_k\}$ the associated partition at the edges $\{p_i,\star\}$ of $v_1$. From Lemma \ref{affirmation} deduce that $B_{v_1}({\{p_i,\star\}})\subseteq p_1$. But ${ P}$ is a partition so it is an equality.
\end{proof}

Thus we proved the following:

\begin{corollary}\label{surjtreeofparti}
The map $\psi$ is surjective onto the set of admissible collections of partitions.
\end{corollary}

\begin{proof} (Theorem \ref{bijpartition})
Remark \ref{iasdmissparti} and Corollary \ref{surjtreeofparti} prove the two first statements. 

As an isomorphism between combinatorial trees is a tree map, we deduce that it maps an arc to an arc, so a branch to a branch, and it follows that two isomorphic combinatorial trees have the same image by $\phi$.

Suppose that two combinatorial trees $T$ and $T'$ have the same image by $\phi$. Let us prove that $T$ and $T'$ are isomorphic. Any internal vertex $v\in T$ has at least three branches and every branch contains at least one element of $X$. Take three elements of $X$ lying in three different branches of $v$. Then there exists a unique internal vertex $v'\in T'$ that has these three elements in different branches. Let us define $M(v):=v'$. As $\phi(T)=\phi(T')$, this does not depend on the choice of the elements in $X$. For $x\in X$ we define $M(x):=x$. 

Let us prove that $M$ is a tree map.
For simplicity we are going to suppose that $T'=T_{\phi(T)}$. Take $v$ and $w$ two adjacent vertices of $T$. Suppose that neither $v$ or $w$ are leaves. Then it is clear that $B_v(w)=X\setminus B_w(v)$. Because $M(v)$ and $M(w)$ give the same partitions of $X$ as $v$ and $w$, the construction of $T_{\phi(T)}$ insures that $M(v)$ and $M(v')$ are adjacent. The case when $v$ or $w$ is a leaf is similar.



\end{proof}


\subsection{Isomorphism of trees of spheres and topology}

\begin{definition}[Isomorphism of trees of spheres]
An isomorphism ${{\cal M}:\T_1\to\T_2}$ between two trees of spheres marked by $X$ is an isomorphism $M:T_1\to T_2$ between the corresponding combinatorial trees and for every internal vertex $v\in T_1$, a (projective) isomorphism $m_v:\S_v\to\S_{M(v)}$ that maps the attaching point of an adjacent edge $e$ to the attaching point of $M(e)$.
\end{definition}


Let $\CST_X$ be the set of trees of spheres marked by $X$. We define on $\CST_X$ an equivalence relation given by : $\T\sim \T'$ if and only if there exists an isomorphism $\mathcal{M}:\T\to\T'$ of trees of spheres marked by $X$. Note that it follows that for all internal vertex $v$ of $T$, $m_v:\S_v\to\S_{M(v)}$ is an isomorphism and $a_{M(v)} = m_v \circ a_v.$
We will sometime use the notation $T\sim_{\mathcal{M}}T'$.

We call moduli space of trees of spheres marked by $X$ and denote by ${\CMT_X}$ the quotient of the set $\CST_X$ by this equivalence relation.
We denote by ${\MT_X}$ the set of isomorphism class of a tree of spheres with a unique internal vertex marked by $X$. 

The moduli space ${\M_X}$ of spheres marked by $X$ is the set of injections of $X$ in $ \S $ modulo post-composition by a Moebius transformation. If $v$ is the unique vertex of $\T\in{\MT_X}$, then $\T$ is determined by the element $[a_v]\in {\M_X}$. Reversely, given an element $[i:X\to\S]\in{\M_X}$, we can define a tree $\T$ marked by $X$ that has a unique internal vertex $v$ with $\S_v=\S$ and $a_v=i$. We will freely consider these elements in ${\M_X}$ or ${\MT_X}$ according to convenience.

 The space space ${\M_X}$ is equipped with a quasi projective variety structure. Indeed, if we choose three distinct points of $X$, we can associate to every element of ${\M_X}$ the set of their cross ratios with the other elements of $X$ and which does not depend on the choice of a representative.

For this method, the three points that we chose plays a particular role. As in \cite{FT},  we avoid this problem by considering $\rm Quad_X$, the set of quadruples of distinct elements of $X$ and considering the embedding : 
$$\Top_X:{\M}_X\to  \S ^{ \rm Quad_X}$$
that associates to $[i]\in{\M}_X$ the collection of the cross ratios $$[i(x_1),i(x_2),i(x_3),i(x_4)]_{(x_1,x_2,x_3,x_4)\in \rm Quad_X}.$$

We are going to use this approach to give to ${\CMT_X}$ a projective variety structure.


Denote by $\rm Trip_{X}$ the set of triples of distinct elements of $X$.
Consider a combinatorial tree $T$ marked by $X$. Take $t:=(x_0,x_1,x_\infty)\in \rm Trip_{X}$. The vertices $x_0,x_1$ and $x_\infty$ are separated by a unique vertex $v_t$.
We say that this vertex separates the triple $t$. 

If this $T$ is the combinatorial tree of a tree of spheres $\T$, then the map $a_{v_t}$ maps the three elements of $t$ to distinct images. So there exists a unique projective chart $\sigma_t:\S_{v_t}\to  \S $ satisfying $\sigma_t\circ a_{v_t}(x_0)=0,\sigma_t\circ a_{v_t}(x_1)=1$ and $\sigma_t\circ a_{v_t}(x_\infty)=\infty$. The map $\sigma_t$ is called {\it the $t$-chart} of $T$. Define \[\alpha_t:=\sigma_t\circ a_{v_t} : X\to  \S .\]
One can easily check that the composition $\sigma_t\circ a_{v_t} : X\to  \S$ does not depend on the choice of a representative in the class of $\T$.




Recall that ${ \rm Quad_X}$ is the set of quadruples of distinct elements of $X$.

\begin{definition}
We define the following map:
$$\TopT_X:{\CMT_X}\to  \S ^{ \rm Quad_X}$$
that  maps every $[T]\in{\CMT_X}$ to the collection of the $(\alpha_t(x))_{(t,x)\in \rm Quad_X}$.
\end{definition}

The map $\TopT_X$ defines a topology on ${\CMT_X}$. We will sometime simply write $\TopT$ when there is no possible confusion.
 The following lemma implies that this topology is Haussdorff.

\begin{lemma}\label{rempart2} The map $\TopT$ is injective.
\end{lemma}

\begin{proof}
Let $T$ be a tree of spheres marked by $X$.
For a fixed $t\in \rm  Trip_{X}$, the data of $\alpha_t(x)$ is sufficient to build the map $a_v$ when $t$ is separated by the vertex $v$ of $T$. Every internal vertex $v$ is adjacent to at least three edges, so we can always find an element of $\rm Trip_{X}$ separated by $v$. Thus, Theorem \ref{bijpartition} assures that the class of $T$ is uniquely determined.
\end{proof}

  \begin{remark}
  It follows that the map $\TopT$ is an homeomorphism onto its image that equips ${\MT_X}$ with a quasi projective variety structure which is the same as the one of ${\M_X}$ (via the identification).
\end{remark}


\subsection{Convergence notion}

We define as in \cite{A1} the notion of convergence of  a sequence of marked spheres to a marked tree of spheres as follows.

\begin{definition}[Convergence of marked spheres]
A sequence of marked spheres $x_n:X\to{\mathbb S}_n$
  converges to a \ts ${\cal T}^X$
  if for all internal vertex $v$
 of $\T^X$ , there exists a (projective) isomorphism $\phi_{n,v}:{\mathbb S}_n\to{\S}_v$
    such that $\phi_{n,v} \circ x_n$
     converges to $a_v$. 
     \end{definition}
We will use the notation $\displaystyle {x}_n\underset{\phi_n}\longrightarrow {\cal T}^X$ or simply ${x}_n\to {\cal T}^X$. 
This convergence notion is compatible with the topology.

\begin{lemma}\label{convquot1} 
Let $({\T}_n)_n$ and $({\T}'_n)_n$ be two sequences of spheres marked by $X$ and let ${\cal T}$ and ${\cal T}'$ be two trees of spheres marked by $X$.  
\begin{enumerate} 
\item\label{quot1} (quotient) \begin{itemize}
	\item if ${\T}\sim {\T}'$, then ${\T}_n\to {\T}\iff {\T}_n\to {\T}'$. 
	\item if ${\T_n}\sim {\T'_n}'$, then ${\T}_n\to {\T}\iff {\T}'_n\to {\T}$.
	\end{itemize}
 \item\label{quot2} (unicity of the limit) if ${\T}_n\to {\T}$ and ${\T}_n\to {\T'}$, then ${\T}\sim {\T}'$. \end{enumerate}
\end{lemma}

\begin{proof} If $\T_n\to_{\phi_n}\T$ and $\T'\sim_{\mathcal{M}}\T$ then $\T_n\to_{\phi'_n}\T'$ with ${\phi_{n,v}=m_v\circ\phi'_{n,v}}$.
In addition, if $\T_n\sim_{\mathcal{M}}\T'_n\to_{\phi'_n}\T$ then $\T_n\to_{\phi'_n\circ{\cal M}}\T$ which concludes the proof of point \ref{quot1}.

For point \ref{quot2}, suppose that $\T_n\to_{\phi_n}\T$ and $\T_n\to_{\phi'_n}\T'$. For every internal vertex $v$ of $\T$, chose a triple $t\in\rm{Trip}_X$ such that $v$ separates $t$. Denote by $v'$ the vertex of $\T'$ separated by the same $t$. The map $\sigma_t'\circ\phi_{n,v'}'^{-1}\circ\phi_{n,v}\circ\sigma_t^{-1}$ is a Moebius transformation that fixes $0,1$ and $\infty$, so it is the identity.
Thus $\phi_{n,v'}'^{-1}\circ\phi_{n,v}$ converges to an isomorphism $m_v$.
\end{proof}

\begin{proposition}\label{ref1}
The map $\TopT$ defines the same convergence notion as the one on trees of spheres on ${\MT_X}$, ie : $$ \T_n \to \T \text{ if and only if }\TopT([\T_n]){{\to}}\TopT([\T]).$$
\end{proposition}

\begin{proof} Lemma \ref{convquot1} assures that these two formulations are equivalent.
Suppose that $\T_n\underset{\phi_n}\longrightarrow \T.$ Let $t\in \rm Trip_{X}$. Let $x\in X$ which does not lie in $t$. Let $\sigma_{n,t}$ be the $t$-chart of $\T_n$. Let $\sigma_t$ be the $t$-chart of $\T$. Let $v$ be the vertex of $\T$ defined by $t$.
Then $m_n:=\sigma_t\circ\phi_{n,v}^{-1}\circ\sigma_{n,t}$ (cf the following diagram) is a Moebius transformation that fixes $0,1$ and $\infty$ so $m_n$ is the identity.\\
\centerline{
$\xymatrix{
 X \ar[r]^{a_n}   \ar@{-->}[rd]_{a_v}  &\S_n \ar[r]^{\sigma_{n,t}}   \ar[d] ^{\phi_{n,v}}& \S   \ar[d]^{m_n} \\
       &  \S_v  \ar[r]_{\sigma_t}  & \S 
  }$}\\
  Then we have $$\sigma_{n,t}\circ a_n (x)=m_n\circ\sigma_{n,t}\circ a_n(x)=\sigma_{t}\circ\phi_{n,v}\circ a_n(x)\to\sigma_t\circ a_v(x).$$
   Thus $\alpha_{n,t}\to\alpha_t$ so  $\TopT([\T_n])\to\TopT([\T]).$
  
If in addition $\TopT([\T_n])\to\TopT([\T])$, for all internal vertex $v$ of $\T$ denote by $t_v$ a triple that defines $v$ and $\sigma_{n,t_v}$ the $t_v$-chart of $\T_n$. Define $\phi_{n,v}:=\sigma_{n,t_v}^{-1}\circ\sigma_{t_v}$. Then we have $\phi_{n,v}\circ a_{n}\to a_v$.
\end{proof}

The notion of convergence of a sequence of trees of spheres is not used in this paper. Let us however make it explicit in the remark below.

\begin{remark}[Convergence of trees]\label{convgeneral} 
Let $(\T_n)_n$ be a sequence of trees marked by $X$ and $\T\in{\CMT_X}$.  For all $t\in \rm Trip_X$ we denote by $v_{n,t}$ the vertex of $\T^X_n$ separating $t$. By the definition of $\TopT$ we know that $(\T_n)_n$ converges to $\T$ if 

$$\forall t\in \rm Trip_X,\exists \phi_{n,v_{n,t}}\in {\rm Aut}(\S_{v_{n,t}},\S_{v_t}),\quad\phi_{n,{v_{n,t}}}\circ a_{n,v_{n,t}}\to a_v.$$ 
\end{remark}

 To conclude this section, we state the following useful property for which we recall a proof from \cite{A1} to make clear the independence with this latter paper.
 
 \begin{lemma} \label{noncomp}
Let $v$ and $v'$  be two distinct internal vertices of $\T^X$,  
and consider a sequence of marked spheres $(x_n)_n$ such that $x_n\underset{\phi_n}\longrightarrow \T^X$. 
Then the sequence of isomorphisms $(\phi_{n,v'} \circ \phi_{n,v}^{-1})_n$
converges locally uniformly outside $i_v(v')$ to the constant $ i_{v'}(v)$. 
\end{lemma}

\begin{proof}
Each vertex $v$ and $v'$ has three edges and every branch has at least a leaf, so there exist four marked points $\chi_1,\chi_2,\chi_3,\chi_4\in X$ 
such that $v$ separates $\chi_1$, $\chi_2$ and $v'$, and the vertex $v'$ separates $\chi_3$, $\chi_4$ and $v$.

We define for $j\in \{1,2,3,4\}$, 
\[\xi_j:=a_v(\chi_j),\quad \xi'_j:=a_{v'}(\chi_j),\quad 
\xi_{j,n}:=\phi_{n,v}\circ x_n(\chi_j)\quad \text{and}\quad 
\xi'_{j,n}:=\phi_{n,v'}\circ x_n(\chi_j).\]
From the hypothesis, $\xi_{j,n}\to \xi_j$ and $\xi'_{j,n}\to \xi'_j$ when $n\to \infty$. Moreover, $\xi_3=\xi_4=i_v(v')$ and $\xi'_1=\xi'_2=i_{v'}(v)$. 
After maybe post-composing $\phi_{n,v}$ and $\phi_{n,v'}$ by some automorphisms of $\S_v$ and $\S_{v'}$ converging to the identity when $n\to \infty$ (so who don't change the limit of $\phi_{n,v'}\circ \phi_{n,v}^{-1}$), we can suppose that for all $n$,
\[\xi_{1,n}=\xi_1,~\xi_{2,n}=\xi_2,~\xi_{3,n}=\xi_3,~\xi'_{1,n}=\xi'_1,~\xi'_{3,n}=\xi'_3\text{ and }\xi'_{4,n}=\xi'_4.\]

Now we consider  the projective charts $\sigma$ on $\S_v$ and $\sigma'$ on $\S_{v'}$ defined by :
\begin{enumerate}
\item $\sigma(\xi_1)=0$, $\sigma(\xi_2)=1$ and $\sigma(\xi_3)=\infty$;
\item $\sigma'(\xi'_1)=0$, $\sigma'(\xi'_4)=1$ and $\sigma'(\xi'_3)=\infty$.
\end{enumerate}
The Moebius transformation $M_n:=\sigma'\circ \phi_{n,v'}\circ \phi_{n,v}^{-1}\circ \sigma^{-1}$ fixes $0$ and $\infty$ and maps $\sigma(\xi_4)$ to $1$. Thus 
\[M_n(z) = \frac{z}{\lambda_n}\quad\text{with}\quad
\sigma(\xi_{4,n}) \underset{n\to \infty}\longrightarrow \infty.\]
Consequently, $M_n$ converges locally uniformly outside infinity to the constant map equal to zero. Thus, $\phi_{n,v'}\circ \phi_{n,v}^{-1}={\sigma'}^{-1}\circ M_n\circ\sigma$ converges locally uniformly to the constant $({\sigma'})^{-1}(0)=i_{v'}(v)$ outside $\sigma^{-1}(\infty)=i_v(v')$. 
 \end{proof}


\subsection{Compactness}

In this section we prove that 
$$\quad\TopT({\CMT_X})=\overline{{\TopT({\MT_X})}}.$$


The proof will be divided into two inclusions (Lemma \ref{ref0} and Lemma \ref{ref00}).

\begin{lemma}\label{ref0}
 The set ${\MT_X}$ is dense in ${\CMT_X}$. In particular we have $$\TopT({\CMT_X})\subseteq \overline{{\TopT({\MT_X})}}.$$
\end{lemma}

In order to prove this, we use the notion of convex hull :

\begin{definition}[Convex hull] For every combinatorial tree $T$ and every set of vertices $V'\subset T$, the convex hull of $V'$ is the sub tree consisting of the arcs connecting the elements of $V'$.
\end{definition}

Note that it is the smallest  subtree of $T$ containing $V'$ (connected hull).

\begin{proof}[Lemma \ref{ref0}] By Proposition \ref{ref1}, the two formulations are equivalents: it is sufficient to show that every tree of spheres marked by $X$ is the limit of spheres marked by $X$. 
Define $X=\{x_1,x_2,\ldots,x_{n_0}\}$. 
For $3\leq k\leq n_0$, define by $X_k:=\{x_1,\ldots,x_k\}$ and denote by $\rm Conv_k$ the set of vertices of valence greater then 3 of the convex hull of $X_k$ in $\T$.
 We prove by recurrence on $k$ that we can find a sequence of spheres $(\T_n)_n$ marked by $X_k$ and for all internal vertex $v\in \rm Conv_k$ a sequence of isomorphisms $\phi_{n,v}:\S_n\to\S_v$ such that $\phi_{n,v}\circ a_n\to a_v$ where $a_n$ will always denote the marking of $\S_n$ which is the sphere of the internal vertex of $\T_n$.
 
 If $k=3$, $\rm Conv_k$ has a unique vertex $v$. Take for all $n\in\N$ a sphere equipped with a complex structure $\S_n$ and some injection $a_n:X_k\to\S_n$. As $X_k$ has only three elements, there exists a unique isomorphism $\phi_{n,v}:\S_n\to\S_v$ such that $\phi_{n,v}\circ a_n$ and $a_v$ are equal on $X_k$. Thus we have $\phi_{n,v}\circ a_n\to a_v$.
 
 Suppose that the property is true for a given $k$ with $3\leq k<n_0$. Denote by $(\S_n)_n$ and $(\phi_{n,v})_{n\in\N,v\in \rm Conv_k}$ the sequences given by the recursive assumption. Let $v_0$ be the vertex of ${\rm Conv}_{k+1}$ which is the closest to $x_{k+1}$ (counting the number of vertices in $[v_0,x_{k+1}]$). 
 
 If $v_0\in \rm Conv_k$ then $\rm Conv_k=\rm Conv_{k+1}$. Define 
 $$U:=\bigcup_{n\in\N}\phi_{n,v_0}\circ a_n(X_k).$$
  As $\phi_{n,v}\circ a_n\to a_v$ and $B_{v_0}(x_{k+1})\cap X_k=\emptyset$, $U$ has a finite number of elements in a small enough neighborhood of $a_{v_0}(x_{k+1})$. Then we can chose an arbitrary sequence $(\zeta_n)_n$ of elements of $\S_{v_0}{\setminus} U$ such that $\zeta_n\to a_{v_0}(x_{k+1})$.
  We define $a'_n:X_{k+1}\to\S_n$ equal to $a_n$ on $X_k$ and such that $a'_n(x_{k+1}):=\zeta_n$. As $\zeta_n\notin U$, the map $a_n$ is an injection and we have $\phi_{n,v}\circ a'_n(x_{k+1})\to a_v(x_{k+1})$. In addition Lemma \ref{noncomp} assures that for every other vertex of $\rm Conv_k$ we have $\phi_{n,v}\circ a'_n(x_{k+1})\to a_v(x_{k+1})$.

  If $v_0\notin \rm Conv_k$, then either $v_0$ lies on an arc between two spheres or there exists a leaf $x\in X_k$ such that $x$ and $v_0$ are adjacent. 
    
 In the first case, take these two spheres $v_1\in B_1,v_2\in B_2$ of $\rm Conv_k$ where $B_1$ and $B_2$ are two distinct branches on $v_0$.
 Define $ X^i=B_i\cap X_k$. We know that $v_1$ lies in an arc $[z_1,z'_1]$ with $z_1,z'_1\in X^1$ and that $v_2$ lies in an arc $[z_2,z'_2]$. We define the triples $t_1:=(z_1,z'_1,z_2)$ separated by the sphere $v_1$  and $t_2:=(z_2,z'_2,z_1)$ separated by the sphere $v_2$. Recall that $\sigma_{t_\star}$ is the $t_\star$-chart of $\T$.
  If we define $M_n:=\sigma_{t_2}\circ\phi_{n,v_2}\circ\phi_{n,v_1}^{-1}\circ\sigma^{-1}_{t_1}$, from the choices of $t_\star$ we have $\forall \xi\in \hat\C,M_n(\xi)=\lambda_n/\xi$ with $\lambda_n\to\infty$.
  

  Define $$U_1:=\bigcup_{n\in\N}\sigma_{t_1}\circ\phi_{n,v_1}\circ a_n(X_k),$$
and $\xi_n:=\sqrt{\lambda_n}+\epsilon$ with $\epsilon\in \C$ independent of $n$ and chosen such that $(\xi_n)_n$ avoids $U_1$.
We define $a_n(x_{k+1}):=\phi_{n,v_1}^{-1}\circ\sigma^{-1}_{t_1}(\xi_n)$. Note that by definition, we have 
$$\phi_{n,v_1}\circ a_n(x_{k+1})\to a_{v_1}(x_{k+1}),$$
  $M_n(\xi_n)=\lambda_n/(\sqrt{\lambda_n}+\epsilon)\to\infty$, and we also have 
 $$\phi_{n,v_2}\circ a_n(x_{k+1})\to a_{v_2}(x_{k+1}).$$ 

Let $\phi_{n,v_0}:\S_n\to\S_{v_0}$ be the unique isomorphism such that $$\phi_{n,v_0}(a_n(z_1))=a_{v_0}(z_1), ~\phi_{n,v_0}(a_n(z_2))=a_{v_0}(z_2)$$ and $$\phi_{n,v_0}(a_n(x))=a_{v_0}(x).$$ 
Define $t:=(z_1,x,z_2)$ and let $\sigma_t$ be the $t$-chart of $\T$. Define $$N_n:=\sigma_{t_0}\circ\phi_{n,v_0}\circ\phi_{n,v_1}^{-1}\circ\sigma^{-1}_{t_1}.$$ We note that $\forall \xi\in\hat\C,N_n(\xi)=\xi/(\xi_n)$. As for all $x\in X^1, \sigma_{t_1}\circ \phi_{n,v_1}\circ a_n(x)$ converges to a finite limit, we have $$\sigma_t=\phi_{n,v_0}(a_n(x))=N_n(\sigma_{t_1}\circ \phi_{n,v_1}\circ a_n(x))\to 0=\sigma_t\circ a_{v_0}(x).$$ By the same kind of considerations on $v_2$, we prove that for every $v\in \rm Conv_k$, from Lemma \ref{noncomp}, we have 
$\phi_{n,v}\circ a'_n(x_{k+1})\to a_v(x_{k+1})$.

If there exists a leaf $x\in X_k$ such that $x$ and $v_0$ are adjacent, then $v_0$ is adjacent to a unique internal vertex $v_1$ of $\rm Conv_k$ and separates the vertices $x,v_0$ and $v_1$. We define $a_n(x_{k+1})$ as a sequence such that $\phi_{n,v_1}\circ a_n(x_{k+1})\to\phi_{n,v_1}(x_{k})$ and such that $a_n|_{X_{k+1}}$ is injective. We conclude as before by taking $\phi_{n,v_0}$ the unique isomorphism mapping the attaching points on $\S_n$ of the branches containing $x,x_k$ and $x'$ to the one of $\S_v$.
\end{proof}
  
 \begin{remark}
 Lemma \ref{ref0} can be proven by gluing spheres minus a finite number of points. This other method is called a  ``plumbing" and this point of view closer to the one of \cite{FT}. We will use this technic for example in the proof of Proposition \ref{revisol}.
 \end{remark}
 
 \begin{lemma}\label{ref00}
The set $\TopT({\CMT_X})$ is closed and $$ \overline{{\TopT({\MT_X})}} \subseteq  \TopT({\CMT_X}).$$
 \end{lemma}
 
\begin{proof}  
Let $(\T_n)_n$ be some sequence of spheres marked by $X$.   
For every $t\in \rm Trip_{X}$, we denote by $\sigma_{n,t}$ the $t$-chart of $\T_n$, then we have $\sigma_{t,n}\circ a_{n,t}$ converges to a map that we will denote by $a_{t}:X\to \hat\C$.  

Every $a_t$ defines a partition $P_t$ of $X$ which are classes of the following equivalence relation: $x\sim x'$ if and only if $a_t(x)=a_t(x')$.
We prove that the collection $\Pt$ of the $P_t$ for $t\in \rm Trip_{X}$ is an admissible set of partitions.

 -Property \ref{mt1}. Elements of $t$ have distinct images so $P_t$ contains at least three elements. 
 
-Property \ref{mt2}.
 Take $P_t\in\Pt$ and $B\in P_t$. By definition, for all element $x\notin B$, we have $a_t(x)\notin a_t(B)=\{\star\}$.
 Let $t_0$ be a triple of points with at least two elements of $B$. 
 From Lemma \ref{noncomp},
 the  partition $P_{t_0}$ has a set $B_0$ containing $X-B$.
  If $B_0=X{\setminus} B$ then we are done. 
  If not, $B_0\cap B\neq\emptyset$. Then we consider another triple $t_1\in {(X{\setminus} B)\times(B_0 \cap B)\times B}$ that contains an edge $B_1$ containing $X{\setminus} B$ but such that $\card(B_1)<\card(B_0)$. We continue until that ${\card(B_i)=\card(X{\setminus} B)}$.
 
-Property \ref{mt3}.
We first note that if $t$ is a triple of elements of $X$ in distinct subsets of $P_{t'}$, then we have $P_t=P_{t'}$. 
 Take $t_1$ and $t_2$ such that there exists $B\in P_{t_1}\cap P_{t_2}$ non empty and suppose by contradiction that $P_{t_1}\neq P_{t_2}$. 
  This provides the existence of a set $B_2\in P_{t_2}\setminus\{B\}$ and $x_1,x_2\in B_2$ such that $x_1$ and $x_2$ are in distinct elements of $P_{t_1}$. As $ P_{t_2}$ has at least three elements we take $x_3\notin B_2\cup B$. Take $x_B\in B$.
  Define $t'_2:=(x_1,x_3,x_B)$ and $t'_1=(x_1,x_2,x_3)$. According to the preceding remark, we have $P_{t_1}= P_{t'_1}$ and $P_{t_2}=P_{t'_2}$.
From Lemma \ref{noncomp}, as $a_{n,t'_2}(x_1)$ and  $a_{n,t'_2}(x_2)$ tend to the same limit, $a_{n,t'_1}(x_3)$ and  $a_{n,t'_1}(x_B)$ too. So, from the point of view of $P_{t_1}$, as we have $x_B\in B$, we deduce that $x_3\in B$ which is a contradiction.
  
According to corollary \ref{propsurj}, the set $\Pt$ determines a unique combinatorial tree (up to isomorphism) and, at each vertex, the associated partition corresponds to the associated partition at an $a_t$. Fix a combinatorial tree $T$ in this isomorphism class and for each of its internal vertices $v$ a triple  $t_v$ such that the partition of $a_{t_v}$ corresponds to the partition of $v$. Define $\phi_{n,v}=a_{n,t_v}$ and $ \S_v= \S $ for every internal vertex $v$. The tree $T$ equipped to the spheres $\S_v$ and the $a_v:=a_{t_v}$ is a tree of spheres $\T$ and by construction we have $\T_n\to_{\phi_n}\T$.
\end{proof}


\subsection{Proof of Theorem \ref{prop519}, comparison with \cite{FT}}

We have to check that the result of the previous section is exactly Theorem \ref{prop519}. We leave this as an easy  exercise  to the reader after recalling below the key steps and the notations of \cite{FT}. (All the labelings below make reference to the ones in \cite{FT}.)

\begin{enumerate}
\item Let $N\in\N$ with $N>3$. The virtual moduli space of type $(0,N)$ is denoted by $\VM(0,N)$ and identified with the set $\VConf(N,\hat\C)$, the set of isomorphism classes of the configuration space of $N$ distinct points on $\hat\C$.  (The set $\VConf(N,\hat\C)$ corresponds to $\M_E$ for $N:=\card(E)$.)

\item Proposition 1.5 states that $\Top|_{\M_E}$ is injective, so $\VConf(N,\hat\C)$ is identified with $\Top(\M_E)$.

\item Definition 1.7 defines a compactification $\VConf(N,\hat\C)^*$ of $\VConf(N,\hat\C)$ to be $\overline{\Top(\M_E)}$.

\item Definition 1.9 introduce a new set $\widehat{\VConf}(N,\hat\C)$ which contains $\VConf(N,\hat\C)$.

\item The set $\widehat{\VConf}(N,\hat\C)$ is identified to $\widehat{\VM}(0,N)$ which has a definition clearly equivalent to the one of ${\CMT}_E$ here. Definition 1.10 defines the topology on $\widehat{\VConf}(N,\hat\C)$ to be the one of $\widehat{\VM}(0,N)$.

\item Theorem 1.13 states that $\VConf(N,\hat\C)^*$ and $\widehat{\VConf}(N,\hat\C)$ are homeomorphic.

\end{enumerate}



\section{Isomorphism classes of covers}\label{chap6}

\subsection{Introduction}
In this subsection we define some notions that one can find with more details in \cite{A1} and we also prove a useful technical lemma.
%
%
%
%
%

The generalization of marked rational maps (defined in the introduction) is the notion of (holomorphic) cover between trees of spheres.
A cover ${\cal F}:{\cal T}^Y\to {\cal T}^Z$ between two trees of spheres marked by $Y$ and $Z$ is the following data
\begin{itemize}
\item a combinatorial tree map $F:T^Y\to T^Z$, 
\item for each internal vertex $v$ of $T^Y$ and $w:=F(v)$ of $T^Z$, an holomorphic ramified cover $f_v:{\S}_v\to {\S}_w$ that satisfies the following properties: 
\begin{itemize}
\item the restriction $f_v : {\S}_v-Y_v\to {\S}_w-Z_w$ is a cover, 
\item $f_v\circ i_v = i_w\circ F$,
\item if $e$ is an edge between two internal vertices $v$ and $v'$, then the local degree of $f_v$ at $i_v(e)$ is the same as the local degree of $f_{v'}$ at $i_{v'}(e)$. 
\end{itemize}
\end{itemize}
In \cite{A1}, it is independently proven that a cover ${\cal F}$ between trees of spheres is such that $F$ is surjective and there is a global degree, denoted by $\deg({\cal F})$.

We will denote by $\RevTC_{\bf F}$ the set of covers between trees of spheres with portrait ${\bf F}=(F|_Y,\deg|_Y)$ and $\RevT_{\bf F}$ the set of covers between two trees that have a unique internal vertex (we respectively talk about covers between trees of spheres marked by ${\bf F}$ and of covers between spheres marked by ${\bf F}$). The set $\RevT_{\bf F}$ is naturally identified with the space $\Rat_{\bf F}$ defined in the introduction.


\subsection{Isomorphisms of covers between trees}

\begin{definition}[Isomorphism between covers]
An isomorphism between two covers between trees of spheres $\F_1:\T_1^Y\to \T_1^Z$ and $ \F_2:\T_2^Y\to \T_2^Z$ is a pair of isomorphisms between trees of spheres $(\mathcal M^Y,\mathcal M^Z)$ such that:
\begin{itemize}
\item$ \T_1^Y\sim_{\mathcal M^Y} \T_2^Y$
and $ \T_1^Z\sim_{\mathcal M^Z} \T_2^Z$;
\item for all the vertices $v_1\in T_1^Y$, $v_2:=M^Y(v_1)\in T_2^Y$, $w_1:=F_1(v_1)\in T_1^Z$ and $w_2:=F_2(v_2)\in T_2^Z$, the following diagram commutes:

\centerline{
$\xymatrix{
    \S_{v_1}\ar[d]_{f_{1,v_1}} \ar[r]^{m^Y_{v_1}}&\S_{v_2}\ar[d]^{f_{2,v_2}}\\
    \S_{w_1}\ar[r]^{m^Z_{w_1}}& \S_{w_2}.\\
  }$}
\end{itemize}
\end{definition}

Thus we write $\F_1\sim \F_2$ or $\F_1\sim_{(\mathcal M^Y,\mathcal M^Z)} \F_2$. As $\mathcal M^Y$ and $\mathcal M^Z$ are invertible, it is an equivalence relation. Equivalence classes of this relation are called  Isomorphism classes of covers between tree of spheres. 

Note that two covers between trees of spheres which are isomorphic have the same degree and same portrait. Thus we can talk about the degree and the portrait of an isomorphism class of covers between trees of spheres. 

We denote by $\revTC_{{\bf F}}$ the quotient of $\RevTC_{{\bf F}}$ by this equivalence relation and $\revT_{{\bf F}}$ the one of $\RevT_{{\bf F}}$.


\subsection{Marked covers, projections and topology}

Recall that $\CST_X$ denote the set of trees of spheres marked by $X$.
Define $$\I:\RevTC_{{\bf F}}\to \CST_Y\times\CST_Z$$ that associate $(\T^Y,\T^Z)$ to $\F:\T^Y\to \T^Z$.
We prove the following proposition by recurrence on the cardinal of $Y$.

\begin{proposition} \label{equivmod}  
The map $\I:\RevTC_{{\bf F}}\to \CST_Y\times\CST_Z$ is an injection.
It descends to an injective map in the quotient:
$$[I]:\revTC_{{\bf F}}\to {\MT}_Y\times{\MT}_Z.$$
\end{proposition}
 
The proof of this proposition follows essentially from the following fact: two maps from the Riemann sphere to itself such that preimages of three distinct points coincide (with multiplicity) are equals.
First we prove the  following lemma: 

\begin{lemma}\label{lemmbout}
Every tree of sphere is either a marked sphere or it has an internal vertex which is adjacent to exactly another one.
\end{lemma}

\begin{proof}
Indeed, consider a leaf and an arc from this leaf which has a maximal number of edges.
If this arc is empty, then there is only one vertex so we don't have to consider this case.  The case where the tree has only two vertices is similar. Thus we suppose that we are not in these cases.

Then the arc has the form $$C=[v_1,v_2,\ldots,v_{k-1} ,v_{k} ].$$with $v_{k-1}\neq v_1$. Note that $v_k$ is necessarily  a leaf because, if not, it will have an edge connecting it to another vertex that allows to extend the arc. 
If $v_{k-1}$ does not satisfies the property then $v_{k-1}$ is adjacent to another internal vertex $v'_k$ that doesn't lie in the arc. As $v'_k$ is an internal vertex, it is also adjacent to a vertex $v'_{k+1}$ that doesn't lie in the arc. Then $C'=[v_1,v_2 ,\ldots,v_{k-1} ,v'_{k} ,v'_{k+1} ]$ would be an arc longer than $C$.  Thus $v_{k-1}$ satisfies the desired property.
\end{proof}

\begin{proof} (Proposition \ref{equivmod})
We prove this result by induction on the cardinal of $Y$. 

We begin with the case $\card(Y)=3$. Take $\F:\T^Y\to \T^Z$ in  $\RevTC_{{\bf F}}$. We prove that $\F$ is uniquely determined by $\I(\F)$.
If $Y$ has only three elements then $T^Y$ has a unique internal vertex $v$. Then $T^Z$ has only one internal vertex $v'$ which is the image of $v$. The combinatorial tree map is uniquely determined. Moreover, as $Z$ has three elements and as we know all their preimages, we know the preimages of three attaching points of three edges on $\S_{v'}$ by $f_{v}$. So $f_{v}$ is also uniquely determined.

Let $Y$ be a set of cardinal $n>3$.
Take $\F:\T^{{Y}}\to \T^{{Z}}$ in $\RevTC_{{\bf F}}$. Suppose that we know $(\T^Y,\T^Z)$ and we prove that $\F$ is uniquely determined. 

If $T^Y$ has only one internal vertex then we do the same proof as before. We suppose that it is not the case.
According to Lemma \ref{lemmbout}, $T^Y$ has an internal vertex $v$ adjacent to a unique internal vertex. Let $y$ be a leaf adjacent to $v$ (it exists because $T^Y$ is stable). The image of $v$ is necessarily adjacent to $z:=F(y)$ which is a leaf. Hence $w:=F(v)$ is uniquely determined. As $T^Y$ has more than one internal vertex, $w$ is adjacent to an internal vertex $w'$. The preimages of $w$ are the vertices adjacent to the preimages of the $z$. Similarly the preimages of $w'$ are all the internal vertices adjacent to $v$. Thus the preimages of $w$ and $w'$ are uniquely determined. 

Now suppose that $v$ is a preimage of $w$. Given that $T^Z$ is stable, $w$ is adjacent to some $z''\in Z-\{z\}$. So we know the preimages of two of its points by $f_v$. Define $e:=\{w,w'\}$. As we know the preimages of $w'$ and of $w$, the preimages by $f_v$ of the attaching point of $e$ on $w$ are the attaching points on $v$ of the edges of $v$ connecting $v$ to some internal vertices. As we know the preimages of three distinct points of $w$ by $f_v$, the map $f_v$ is uniquely determined. 

Thus the preimage by $F$ of $B:=B_{w'}(e)$ is uniquely determined. Define $T'':=V^Z{\setminus} B$ and $T':=T^Y{\setminus} F^{-1}(B)=F^{-1}(T'')$. Now we prove that $F|_{T'}$ is uniquely determined. 

Let $E\subset T^Z$ be the set of the leaves adjacent to $w$ and of the the edges adjacent to them. The graph $T'':=T^Z\setminus E$ is clearly a tree and we know the preimages of its leaves by $F$. Let $T'$ be any connected component of $T^Y$ minus the preimages of $E$. This is a tree whose leaves are the elements of $Y\setminus F^{-1}(E\cap Z)$ and the preimages of $w'$, thus we know the map $F$ on its set of leaves. We define $\T''$ to be the tree of spheres of combinatorial tree $T'$ and for which the spheres associated to the internal vertices and the attaching points of edges are the same as the one for $\T^Y$. Similarly we define $\T''$ from $T'$ and $\T^Z$.

It is easy to check that the natural restriction $\F':\T'\to\T''$ of $\F$ is a cover between trees of spheres. As the set of leaves of $T''$ has less elements than the one of $T^Y$ and as we know the portrait of $\F'$, the induction allows reconstruct $\F:\T'\to\T''$ from the pair $(\T',\T'')$.
\end{proof}

Denote by $\pi_1$ the projection on the first coordinate. 

\begin{definition}
We define the map $$\MI: \revTC_{{\bf F}}\to {\CMT}_Y
\text{ by setting }   \MI:=\pi_1\circ[I].$$
\end{definition}

\begin{proposition}\label{Iinject}
The map $\MI: \revTC_{{\bf F}}\to {\CMT}_Y$ is injective.
\end{proposition}

\begin{proof} Take $\F:\T^Y\to\T^Z$  in $\RevTC_{{\bf F}}$. Let $v_0$ be a vertex given by Lemma \ref{lemmbout}. Let $v'_0$ be its image. Let $V_0$ be the set of the leaves adjacent to $v_0$. The portrait $({ F}, \deg)$ determines the images of the elements of $V_0$ that have to be adjacent to $v'_0$. The other preimages of $v_0$ are adjacent to the elements of ${ F}^{-1}\circ{ F}(V_0)$. As we did in the proof of Proposition \ref{equivmod}, we determine the vertices of the tree $T^Z$ and the map $F$ from the data of $T^Y\setminus F^{-1}(v'_0)$ by induction on the number of vertices of $T^Y$.

Thus, it is possible to reconstruct   the combinatorial tree $T^Z$ and the combinatorial tree map from $\T^Y$. We prove that the attaching points of the edges of $T^Z$ on the vertices of $\T^Z$ are well determined up to post-composition by automorphisms. For this, it is sufficient to show that for each internal vertex $v$ of $T^Z$, the attaching points of $E_v$ are determined by the data of three of them.

For every internal vertex $v$ of $T^Z$, we suppose that we know the attaching points $z_0,z_1$ and $z_\infty$ of three distinct edges $e_0,e_1,e_\infty$ on $v$. For every preimage $w$ of $v$, there exists a unique holomorphic cover $ f_w:\S_w\to \S_v$ mapping the preimages of the edge $e_0$ (resp. $e_1,e_\infty$) on $z_0$ (resp. $z_1,z_\infty$). If $e$ is an edge on $v$ then $e$ has a preimage $e'$ on $w$ so its attaching point has to be $f_w(e'_w)$.
\end{proof}

We define a topology on the set of isomorphism classes of the covers between trees of spheres via the map $\MI$. 


\subsection{Convergence notion}

We define the notion of convergence of a sequence of marked spheres covers to marked cover between trees of spheres as follows.

\begin{definition}[Non dynamical convergence]
 A sequence of ${{ \F}_n:=(f_n,a_n^Y,a_n^Z)\in \RevT_{{\bf F}}}$ converges to ${ \F}:{\T}^Y\to { \T}^Z$ be in $ \RevTC_{{\bf F}}$ if and if for all pair of internal vertices $v$ and $w:=F(v)$, there exists sequences of isomorphisms $\phi_{n,v}^Y:\S_n^Y\to \S_v$ and $\phi_{n,w}^Z:\S_n^Z\to \S_w$ such that 
\begin{itemize}
\item $\phi_{n,v}^Y\circ a_n^Y:Y\to \S_v$ converges to $a_v^Y:Y\to \S_v$, 
\item $\phi_{n,w}^Z\circ a_n^Z:Z\to \S_w$ converges to $a_w^Z:Z\to \S_w$ and 
\item $\phi_{n,w}^Z\circ f_n\circ (\phi_{n,v}^Y )^{-1}:\S_v\to \S_w$ converges locally uniformly outside $Y_v$ to ${f_v:\S_v\to \S_w}$. 
\end{itemize}
\end{definition}

We use the notation $\F_n\rightarrow \F$ or $\F_n\underset{(\phi^Y_n,\phi^Z_n)}\longrightarrow  \F.$

Let us prove that this topology is compatible with the convergence notion.

\begin{lemma}\label{convquot} 
Let $(f_n)_n$ and $(f'_n)_n$ be two sequences in $ \RevT_{{\bf F}}$. For ${\F}$ and $\F'$ be in $ \RevTC_{{\bf F}}$ we have: 
\begin{enumerate} 
\item\label{quot111} (quotient) \begin{itemize}
	\item if ${\F}\sim {\F}'$, then ${f}_n\to {\F}\iff {f}_n\to {\F}'$. 
	\item if ${f_n}\sim {f'_n}$, then ${f}_n\to {\F}\iff {f'}_n\to {\F}$.
	\end{itemize}
 \item\label{quot222} (unicity of the limit) if ${f}_n\to {\F}$ and ${f}_n\to {\F'}$, then ${\F}\sim {\F}'$. \end{enumerate}
\end{lemma}

\begin{proof}
If $f_n\underset{(\phi^Y_n,\phi^Z_n)}\longrightarrow\F$ and $\F\sim_{(\mathcal M^Y,\mathcal M^Z)} \F'$ then $$f_n\underset{(\psi^Y_n,\psi^Z_n)}\longrightarrow\F'\text{ with }\psi^\star_{n,v}:=M^{\star}_v\circ\phi^\star_{n,v}.$$ 
Moreover, suppose that ${f'_n}\sim_{(\mathcal M_n^Y,\mathcal M_n^Z)} {f_n}$, $$\text{if }f_n\underset{(\phi^Y_n,\phi^Z_n)}\longrightarrow\F\text{ then }{f'_n\underset{(\mathcal M^Y\circ\phi^Y_n,\mathcal M^Y\circ\phi^Z_n)}\longrightarrow\F}.$$
So this convergence notion is well behaved under the quotient.

For \ref{quot222}, we suppose that $f_n\underset{(\phi^Y_n,\phi^Z_n)}\longrightarrow\F$ and $f_n\underset{(\psi^Y_n,\psi^Z_n)}\longrightarrow\F'$ then $${\F\sim_{(\mathcal M^Y,\mathcal M^Z)} \F'}\text{ with }m^{\star}_v:=\lim_{n\to\infty}\psi^\star_{n,v}\circ(\phi^\star_{n,v})^{-1}.$$
Indeed, $(a^\star_{n,v})_n$ converges to $a^\star_v$ and $((\psi^\star_{n,v}\circ(\phi^\star_{n,v})^{-1})^\star\circ a^\star_{n,v})_n$ converges to $a'^\star_v$ on $Y$ which contains at least three points, so $m_v$ is an isomorphism.
\end{proof}

\begin{corollary}\label{sensdirect}
The convergence notion defined on $\RevTC_{\bf F}$ implies the one given by the topology given by $\MI$ :
$$\text{if } f_n \to \F \text{ then }\MI([f_n]){{\to}}\MI([\F]).$$
\end{corollary}

\begin{proof}
Indeed,  if $ (f_n:\T^Y_n\to\T^Z_n) \to (\F:\T^Y\to\T^Z)$, then by definition we have $\T^Y_n\to\T^Y$, ie $I( f_n) \to I(\F)$ so in the quotient $\MI([f_n]){{\to}}\MI([\F]).$
\end{proof}

We will prove the reciprocal property in the next section.

Here is a nice property to note:

\begin{lemma}\label{cvu}
Let ${\cal F}:{\cal T}^Y\to{\cal T}^Z$ be in $ \RevTC_{{\bf F}}$. Let
$v$ be an internal vertex of $T^Y$ with $\deg(v)=\deg \F$ and let ${\cal
F}_n:=(f_n,a_n^Y,a_n^Z)\in  \RevT_{{\bf F}}$ such that $\displaystyle {\cal F}_n\underset{\phi^Y_n,\phi^Z_n}\longrightarrow{\cal F}$. Then the sequence $\phi_{n,F(v)}^Z\circ f_n\circ (\phi_{n,v}^Y)^{-1}:\S_v\to \S_{F(v)}$ converges uniformly
to $f_v:\S_v\to \S_{F(v)}$.
\end{lemma}

This property is not used in this paper but the interested reader can find a proof in \cite{A1}.


\subsection{Compactness}

In this section we prove Theorem \ref{compseqrev} by the inclusions \begin{itemize}
\item ${\overline{{\MI}(\revT_{{\bf F}})}\subseteq\MI(\revTC_{{\bf F}})}$ in Proposition \ref{adhrev2}, and 
\item $\MI(\revTC_{{\bf F}})\subseteq\overline{{\MI}(\revT_{{\bf F}})}$ in propositions \ref{revisol}.
\end{itemize}

First note the fundamental result:

\begin{lemma}\label{distKoeb}
Let $(f_n: \S \to  \S )_n$ be a sequence of rational maps of same degree. Then, there exists a subsequence $(f_{n_k})_{n_k}$ and a sequence of Moebius transformations $(M_{n_k})_{n_k}$ such that $ (M_{n_k}\circ f_{n_k})_{n_k}$ converges to a non constant rational map $f$  uniformly outside a finite number of points.
\end{lemma}

\begin{proof}
Define $x_0=\infty$. We extract a subsequence in order to have $${X_n:=f_n^{-1}\circ f_n(x_0)\to X}$$ with multiplicity.
Choose $y_0\in\C\setminus X$.  We extract a subsequence in order to have $$Y_n:=f_n^{-1}(f_n(y_0))\to Y$$ with multiplicity.
Choose $z_0\in\C\setminus (X\cap Y)$. Again, we extract a subsequence in order to have $$Z_n:=f_n^{-1}(f_n(z_0))\to Z.$$

By construction, for all $n$ we can find a Moebius transformation satisfying:
$$ M_n\circ f_n(x_0)=\infty,\;      M_n\circ f_n(y_0)=0,   \; M_n\circ f_n(z_0)=1.$$

Thus we have $$\forall w\in\C, M_n\circ f_n(w)=\frac{\prod_{x\in X_n}(w-x)}{\prod_{y\in Y_n}(w-y)} .\frac{\prod_{y\in Y_n}(z_0-y)}{\prod_{x\in X_n}(z_0-x)}.$$
This sequence of rational maps converges uniformly to a non constant rational map outside a finite number of points which correspond to $X\cap Y$.
\end{proof}

  \begin{proposition}\label{adhrev1}  
Let $(\F_n)_n$ be a sequence in $\RevT_{{\bf F}}$. If $(\MI([\F_n]))_n$ converges in  $\S ^{ \rm Quad_Y}$ then $(\F_n)_n$ converges to a cover between trees of spheres $\F$.
 \end{proposition}

\begin{proof}
Let $(\F_n:\T^Y_n\to\T^Z_n)_n$ be a sequence of element of $\RevT_{{\bf F}}$ such that $([I] (\F_n))_n$ converges in $\CST_Y$.
Thus by definition$\T^Y_n\to_{\phi^Y_{n}}\T^Y\in\CST_Y$.  

Fix some charts $\sigma_n^Z:\S^Z_n\to \S $. For every internal vertex $v$ of $\T^Y$, we set $\tilde f_{n,v}:=\sigma^Z_{n}\circ f_n\circ  (\phi^Y_{n,v})^{-1}$.
 According to Lemma \ref{distKoeb}, after passing to a subsequence, we can find a sequence of isomorphisms $(M_{n,v}: \S \to \S )_n$ such that  $(M_{n,v}\circ \tilde f_{n,v})_n$ converges uniformly outside a finite number of points to a non constant holomorphic morphism $\tilde f_v:\S_v\to \S $. (Passing to a subsequence does not affect the result here because we are going to find a $\F$ that depends only on $\T^Y$.)
 We set \begin{itemize}
 \item $\sigma^Z_{n,v}:=M_{n,v}\circ\sigma^Z_{n}$;
\item $\tilde a_v=\lim_{n\to\infty}~ \sigma^Z_{n,v}\circ a^Z_{n}$;
\item ${Y}_v=a_v(Y)$ and  $\tilde{Z}_v=\tilde a_v (Z)$. 
\end{itemize}
 Note that $\tilde f_{v}(Y_v)=\tilde a_v (Z)$.

\centerline{
$\xymatrix{
&\S_n^Y\ar[rr]^{f_n}\ar[dd]_{\phi^Y_{n,v}}\ar[rrdd]^{\tilde f_{n,v}}&&\S_n^Z\ar[dd]^{\sigma_{n,v}}&\\
Y\ar[ru]^{a_n^Y}\ar@{-->}[rd]_{a_v}&&&&Z\ar[lu]_{a^Z_n}\ar@{-->}[ld]^{\tilde a_v}\\
&\S_v\ar@{-->}[rr]^{\tilde f_v}&& \S &
  }$}

\begin{claim}
      Let $\gamma_z$ be the boundary of a small disk around $z\in \tilde{Z}_v$.
      Let $y\in {Y}_v$ be such that $\tilde f_v(y)=z$.
Then there exists $\gamma_y$ surrounding $y$ such that $\tilde f_v(\gamma_y)=\gamma_z$ and
          $\tilde f_{n,v}(\gamma_y)\to\gamma_z$.
\end{claim}

\begin{proof} Indeed, if $\gamma_z$ is small enough, $\tilde f_v^{-1}(\gamma_z)$ is a loop $\gamma_y$ which is the boundary of a disk  containing $y$ and avoiding the other elements of $Y_v$. As on $\S_v\setminus Y_v$, the convergence is uniform, so $\tilde f_{n,v}(\gamma_y)\to\gamma_z$.
\end{proof}

\begin{claim}
For every internal vertex $v$ of $\T^Y$, we have $\card \tilde Z_v\geq 3$.
\end{claim}

\begin{proof} Consider small disks around the elements of $\tilde Z_v$. 
 Suppose $n$ large enough such that the $\tilde f_{n,v}(Y)$ are in these disks. Denote by $D_Z$ the set $ \S $ minus these disks and $D_Y:=\tilde f_v^{-1}(D_Z)$. The Riemann-Hurwitz formula gives 
$$-1\geq\chi(D_Y)=\deg(\tilde f_{v})\chi(D_Z)$$ because $\S_v$ has at least three edges and $D_Y$ has no critical points. As  $\deg(\tilde f_{v})\geq1$, then $\chi(D_Z)\leq -1$ so $\card \tilde Z_v\geq 3$.
\end{proof}

Thus $\card(\tilde a_v(Z))\geq3$. 
Let $t_{v'}$ be a triple of points of $Z$ which have pairwise distinct images by $\tilde a_v$. Let $v'$ be the unique vertex of $T^Z$ separating $t_{v'}$. 
As on the diagram below we use the notation $\sigma_{n,v'}:=\phi^Z_{n,v'}\circ\sigma_{n,v}^{-1}$. 

\centerline{
$\xymatrix{
\S_n^Y\ar[rrd]^{\tilde f_{n,v}}\ar[dd]_{\phi_{n,v}^Y}\ar[rrr]^{f_n}&&&\S_n^Z\ar[ld]_{\sigma_{n,v}}\ar[dd]^{\phi_{n,v'}^Z}\\
&& \S \ar[rd]_{\sigma_{n,v'}}&\\
\S_v\ar@{-->}[rru]^{\tilde f_v}\ar[rrr]^{f_{n,v}}\ar@{-->}@/_1pc/[rrr]^{f_v}&&&\S_{v'}
 }$\medskip}
 From the choice of $t_{v'}$, we know that $\sigma_{n,v'}$ converges to an isomorphism $\sigma_{v'}$. 
Thus $\sigma_{n,v'}\circ\tilde f_{n,v}\to\sigma_{v'}\circ\tilde f_v:=f_v$ locally uniformly outside a finite number of points and $\deg(f_v)\geq1$.

So we have $f_{n,v}:=\phi^Z_{n,v'}\circ f_n\circ(\phi^Y_{n,v})^{-1}\to f_v$  locally uniformly outside a finite number of points and $\deg(f_v)\geq1$.

\begin{claim}
      The map $F:V^Y\to V^Z$ that maps $v$ defined by $t_{v}$ to the vertex defined by $t_{v'}$ extends to a map between trees.
\end{claim}

\begin{proof} 
Let $v_1$ and $v_2$ be two adjacent vertices in $T^Y$ connected by an edge $e$ and let $v_1'$ and $v_2'$ be their respective images. Let $D_1$ (reps. $D_2$) be a topological disk neighborhood of the attaching point of $e$ on $v_1$ (reps. ${v_2}$) and containing only this attaching point of edge and let $C_1$ (resp. $C_2$) be its boundary.  Denote by $A_n:=(\phi^Y_{n,v_1})^{-1}(D_1)\cap(\phi_{n,v_2}^Y)^{-1}(D_2)\subset \S_n$, Denote by $C'_\star:=f_{v_\star}(C_\star)$ and $A'_n:=f_n(A_n)$. We now suppose that $n$ is large enough such that $A_n$ is an annulus and does not contain any attaching point of edges. Thus $A'_n$ does not contain any attaching point of edges neither. As the critical points of $f_n$ are attaching points of edges, $A_n$ does not contain critical points and $A'_n$ is an annulus.

Suppose that there is a vertex $v'$ between $v'_1$ and $v'_2$. As $\phi^Z_{n,v_\star}(C'_\star)\to C'_\star$, Lemma \ref{noncomp} implies that $\phi^Z_{n,v'}(A'_n)$ tends to $\S_{v'}$ minus the attaching points of the branches containing respectively $v'_1$ and $v'_2$. As $A'_n$ does not contain attaching points of edge, $\S_{v'}$ has only two attaching points of edges which contradicts the stability of $T^Z$. Thus $F$ maps two adjacent vertices to two adjacent vertices.
\end{proof}

In particular we proved that the image of the attaching point of $e$ on $v_\star$ is the attaching point of $F(e)$ on $v'_\star$ ie $f_{v_\star}(e_{v_\star})$.

\begin{claim}
      The map $\F:\T^Y\to \T^Z$ defined by $F$ and the $f_v$ is a cover between trees of spheres.
\end{claim}

\begin{proof}
Take $v'_1:=F(v_1)$ with $v_1$ an internal vertex of $T^Y$. Let $e':=\{v'_1,v'_2\}$ be an edge of $T^Z$. Let $C'_1$ (resp. $C'_2$) be a topological circle surrounding a disk $D'_1$ (resp. $D'_2$) containing a unique attaching point on $v'_1$ (resp. on $v'_2$), the one of $e'$. Define $A'_n=\phi^Z_{n,v'_1}(D'_1)\cap\phi^Z_{n,v'_2}(D'_2)$ and suppose $n$ big enough such that $A'_n$ is an annulus. Let $A_n$ be a connected component of $f_n^{-1}(A'_n)$. From the Riemann-Hurwitz Formula, we deduce that $A_n$ is an annulus. Denote by $C_{1,n}$ and $C_{2,n}$ the preimages of $C'_{1}$ and $C'_{2}$ surrounding $A_n$ and by $D_{1,n}$ the disks bounded by $C_{1,n}$ containing $A_n$. We suppose $n$ large enough such that the partition of $a_n(Y)$ (resp. $a_n(Z)$) given by the two connected components of $\S_n{\setminus} A_n$ (resp. $\S_n{\setminus} A'_n$) is constant.

Take $z_2\in Z\cap B_{v_1'}(F(e))$ and $z_1\in Z\cap B_{v_2'}(F(e))$. Then $a_n(z_1)$ and $a_n(z_2)$ are respectively in each of the two connected components of $\S_n{\setminus} A'_n$. After choosing a projective chart $\sigma_n$ such that $\sigma_n\circ a_n(z_1)=0$ and $\sigma_n\circ a_n(z_2)=\infty$, we suppose that $\S_n= \S $, $a_n(z_1)=0$ and $a_n(z_2)=\infty$.

Denote by $$n_0:=\card\{y\in Y\cap D_{1,n}~|~ f_n(y)=0\}$$ and 
$$n_\infty:=\card\{y\in Y\cap D_{1,n}~|~ f_n(y)=\infty\}.$$
The local degree of $f_{v_1}$ at the attaching point of $e$ is the same as the degree of $f_{v_1}$ on $C_{1,n}$ which is the one of $f_n$ on $(\phi^Y_{n,v_1})^{-1}(C_{1,n})$, ie 
$$\deg_{f_{v_1}}(e)=n_0-n_\infty.$$
Note that these two cardinals don't depend on the choice of the pair $(z_1,z_2)$ in the connected components of $\S_n{\setminus} A'_n$. Again these cardinals are the same if we consider $D_2$ instead of $D_1$ because $A_n$ does not contain critical values. By the same deductions on $v_2$ we prove that $\deg_{f_{v_1}}(e)=\deg_{f_{v_2}}(e)$.

In particular, if $n_0\neq 0$ then $\phi_{n,v_1}(D_{1,n})$ contains an attaching point of an edge; thus every preimage of an edge attaching point is the attaching point of an edge. As the image of an edge attaching point is an edge attaching point, $f_v: Y_v\to Z_{F(v)}$ is a cover. Moreover the critical points of $f_v$ are the limits of the critical points of $\phi^Z_{n,F(v)}\circ f_n\circ(\phi^Y_{n,v})^{-1}$ so they are attaching points of edges.
\end{proof}

This concludes the proof of proposition \ref{adhrev1} because as required we have $$\F_n\underset{(\phi^Y_n,\phi^Z_n)}\longrightarrow  \F.$$\end{proof}

\begin{corollary}\label{equivtopo2}
The topology given by $\MI$ is compatible with the convergence notion defined on $\RevTC_{\bf F}$:
$$ \F_n \to \F \quad\text{ if and only if }\quad\MI([\F_n]){{\to}}\MI([\F]).$$
\end{corollary}

\begin{proof}
The implication is given by Corollary \ref{sensdirect}.
Reciprocally if $$\MI([\F_n:\T^Y_n\to\T^Z_n]){{\to}}\MI([\F:\T^Y\to\T^Z])$$ then according to Proposition \ref{adhrev1},
$(\F_n)_n$ converges to a cover between trees of spheres $\F'$ so $\MI([\F_n]){{\to}}\MI([\F'])$.
We deduce that ${\MI([\F'])=\MI([\F])}$, thus $\F=\F'$ according to Proposition \ref{Iinject}.
\end{proof}

We can also directly deduce the theorem assumed in \cite{A1}.

\begin{corollary}\label{thmcomp00}
Let $y_n$ and $z_n$ be two sequences of spheres marked respectively by the finite sets $Y$ and $Z$ containing each one at least three elements and converging to the trees of spheres $\T^Y$ and $\T^Z$. 

Every sequence of marked spheres covers $(f_n,y_n,z_n)_n$ of a given portrait converges to a cover between the trees of spheres $\T^Y$ and $\T^Z$.
\end{corollary}

Hence we have proven:

\begin{proposition}\label{adhrev2} 
The set $\MI(\revTC_{{\bf F}})$ is closed. In particular $$\overline{{\MI}(\revT_{{\bf F}})}\subseteq\MI(\revTC_{{\bf F}}).$$
\end{proposition}

Let us now prove the reverse inclusion.

\begin{proposition}\label{revisol}
 The set $\RevT_{{\bf F}}$ is dense in $\RevTC_{{\bf F}}$. In particular we have $$\MI(\revTC_{{\bf F}})\subseteq\overline{{\MI}(\revT_{{\bf F}})}.$$
\end{proposition}

\begin{proof}
Take $\F:\T^Y\to\T^Z$ in $\RevTC_{{\bf F}}$.
In this proof, we don't distinguish the spheres at the vertices of $\T^Y$ and $\T^Z$ and some of their projective charts arbitrarily chosen. 
Take $1>\epsilon>0$. Take an edge $e$ between two vertices $v_1,v_2$. Define by $v'_i:=F(v_i)$ and denote by $e':=F(e)$ the edge between $v'_1$ and $v'_2$. 

Let $A'_1$ (resp. $A'_2$) be an annulus between the circles of radii $\epsilon^2$ and $\epsilon$ centered on the attaching point $e'_{v'_1}$ (resp. $e'_{v'_2}$). Let $\phi^{e'}_\epsilon:A'_1\to A'_2$ be an biholomorphism that exchanges the borders of the two annuli (maps the circle of radius $\epsilon^2$ on $A'_1$ to the one of radius $\epsilon$ on $A'_2$, reversing the orientation, and reciprocally).

Let $A_i$ be the preimage of $A'_i$ on $v_i$. We consider an $\epsilon$ small enough such that the $A_i$ are in neighborhoods of the $e_{v_i}$ that map with degree $\deg_{f_{v_i}}(e_{v_i})$ and such that each of these neighborhoods contain a unique edge attaching point. As $\F$ is a covering between trees of spheres, we have $d_e:=\deg_{f_{v_1}}(e_{v_1})$=$\deg_{f_{v_2}}(e_{v_2})$. We choose one of the $d_e$ biholomorphisms $\phi^{e}_\epsilon$ that makes the following diagram commuting: 

\centerline{
$\xymatrix{
A_1\ar[r]^{\phi^{e}_\epsilon}\ar[d]_{f_{v_1}}&A_2\ar[d]^{f_{v_2}}\\
A'_1\ar[r]_{\phi^{e'}_\epsilon}&A'_2.
  }$}

As $F:E^Y\to E^Z$ is surjective, after repeating this process we obtain some families $\Phi$ of biholomorphisms associated to the edges between the internal vertices of $T^Y$ and $\Phi'$ associated to the same one of $T^Z$. We suppose $\epsilon$ small enough such that all the annuli already defined don't have common pairwise intersections. For every internal vertex $v$ of $T^\star$, denote by $\S^\star_{\epsilon,v}$ the sphere $\S^\star_v$ minus some topological closed disks around the attaching points of edges connecting to internal vertices which are bordered by the $A_i$ (resp. $A'_i$) as previously defined (but does not contain the $A_i$ (resp. $A'_i$).
We use the notations $$\S^Y_\epsilon:=\bigsqcup_\Phi \S^Y_{\epsilon,v} \text{ and } \S^Z_\epsilon:=\bigsqcup_{\Phi'} \S^Z_{\epsilon,v}.$$

Every element $y$ of $Y$ is a vertex of $T^Y$ which has a unique edge so it is adjacent to a unique internal vertex $v_y$ of $T^Y$. Denote by $e_y$ the attaching point of this edge on $v_y$. 
We define a family of injections $a_\epsilon^Y:Y\to\S^Y_\epsilon$ that associate $v_y$ to $y$. 

\begin{claimint} 
For $\epsilon$ small enough, $\S^Y_\epsilon$ with $a^Y$ is a marked sphere $\T^Y_\epsilon$ and if $\epsilon\to 0$, we have $$\T^Y_\epsilon\to \T^Y.$$
\end{claimint}

\begin{proof}
The set of internal vertices of $T^Y$ and edges connecting them is a subtree $T'$ of $T^Y$. Thus it satisfies  $\card{V'}=\card{E'}+1$ (see for example \cite[Corollary 1.5.3]{graphtheory}).
In addition the Euler characteristic of $\S^Y_\epsilon$ is equal to the sum of the one of the $\check \S^{\epsilon}_{v}$ because the one of an annulus is 0. But the $\check \S^{\epsilon}_{v}$ are spheres minus a disk for each of the  edge of $v\in T'$. So the Euler characteristic of $\S^Y_\epsilon$ is $$\sum (2- \card E'_v)=2\card V'-2\card E'=2(\card V'-\card E')=2.$$
But $\T^Y$ is connected, so $\S^Y_\epsilon$is connected too and it follows that $\S^Y_\epsilon$ is a topological sphere. As $\Phi$ is a family of isomorphisms, $\S^Y_\epsilon$ is equipped of a complex structure. Thus we proved that $\S^Y_\epsilon$ together with $a_\epsilon^Y$ is a sphere marked by $Y$ that we will denote by $\T^Y_\epsilon$.

Moreover, for every $v$ internal vertex of $T^Y$, if we define $\phi_{\epsilon,v}$ an isomorphism defined by the identity on $\check \S^{\epsilon}_{v}$, then we have $\T^Y_{\epsilon}\to_{\phi_\epsilon}\T^Y$ as required because the $\check \S^{\epsilon}_{v}$ tend to the $\S_v$.
\end{proof}

Similarly  we construct a family of injections $a_\epsilon^Z:Z\to\S^Z_\epsilon$ then the associated trees of spheres $\T^Z_\epsilon$ and we have $\T^Z_{\epsilon}\to\T^Z$.

We are now ready to prove that the maps $\F_\epsilon:=(\F|_{\S^Y_{\epsilon}}: \T^Y_{\epsilon}\to\T^Z_{\epsilon})$ form a family of covers between marked spheres (for $\epsilon$ small enough) and $[\F_\epsilon]\to[\F]$.

Indeed, for $\epsilon$ small enough, the $ \S^Y_{\epsilon,v}$ for $v$ internal vertex of $T^Y$ form a cover of $\S^Y_{\epsilon}$ and the map $\F_\epsilon$ restricted on these ones is  holomorphic, then $f_\epsilon$ is holomorphic. By definition $(F_\epsilon|_Y,\deg|_Y)={\bf F}$ so $f_\epsilon$ is a cover on the edges. Thus, for $\epsilon$ small enough, $\F_\epsilon$ is a cover between marked spheres. In addition we have $[\T^Y_{\epsilon}]\to [\T^Y]$.

\end{proof}



\section{Dynamics}\label{chap7}

\subsection{Introduction}

In this section we suppose that $X\subseteq Y\cap Z$ is a finite set with at least three elements. We will say that $({\cal F},{\cal T}^X)$ is a dynamical system between trees of spheres if :
\begin{itemize}
\item ${\cal F}:{\cal T}^Y\to {\cal T}^Z$ is a cover between trees of spheres, 
\item ${\cal T}^X$ is a tree of spheres compatible with ${\cal T}^Y$ and ${\cal T}^Z$, ie : 
	\begin{itemize}
	\item $X\subseteq Y\cap Z$, 
	\end{itemize}
and for  each internal vertex $v$ of $T^X$:
	\begin{itemize} 
	\item $v$ is an internal vertex common to $T^Y$ and $T^Z$, 
	\item  ${\S}_v^X = {\S}_v^Y = {\S}_v^Z$, and 
	\item $a_v^X = a_v^Y|_X = a_v^Z|_X$. 
	\end{itemize}
\end{itemize}

Dynamical covers between marked spheres can be naturally identified to dynamically marked rational maps:

Recall that a rational map dynamically marked by  $({\bf F},X)$ is a rational map $(f,y,z)$ marked by  ${\bf F}$ such that  $y|_X=z|_X$. 
We denote by 
 $\Rat_{{\bf F},X}$ the set of rational maps dynamically marked by $({\bf F},X)$ and by $\rat_{{\bf F},X}$ its quotient modulo the action by conjugacy of the group of Moebius transformation.


\subsection{Conjugacy and compactification}

\begin{definition}
Two dynamical systems between trees of spheres $({\F_1}, \T_1^X)$ and $ (\F_2,\T_2^X)$ are conjugated if there exist two isomorphism between trees of spheres $\mathcal M^Y:\T^Y_1\to \T^Y_2$ and $\mathcal M^Z:\T^Z_1\to \T^Z_2$ such that:
$$\F_1\sim_{(\mathcal M^Y,\mathcal M^Z)} \F_2$$
and for every internal vertex $v\in \T_1^X$, $$m_v^Y=m_v^Z .$$
\end{definition}

We denote by 
$\DRevTC_{{\bf F},X}$ the set of such dynamical systems between trees of spheres of portrait ${\bf F}$.
We denote by 
$\drevTC_{{\bf F},X}$ their conjugacy classes.
With this definition the set $rat_{{\bf F},X}$ defined in the introduction is naturally identified to the set of classes of dynamical systems between marked spheres.

\begin{lemma}
The map that associate to every $[(\F,\T^X)]\in\drevTC_{{\bf F},X}$ the element $[\F]\in\revTC_{{\bf F}}$ is an injection.
\end{lemma}

\begin{proof}
Take $(\F_1,\T^X_1)$ and $(\F_2,\T^X_2)$ in $\DRevTC_{{\bf F},X}$ such that ${\F^1\sim_{(\mathcal M^Y,\mathcal M^Z)} \F^2}$. We want to prove that $[(\F_1,\T^X_1)]=[(\F_2,\T^X_2)]$. 

It is clear that $M^Y|_{T_1^X}=M^Z|_{T_1^X}$.
Take $v$ an internal vertex of $T^X_1$.
As $(\F_2,\T^X_2)\in\DRevTC_{{\bf F},X}$ we have $$a_{M^Y(v)}^X = a_{M^Y(v)}^Y|_X = a_{M^Z(v)}^Z|_X$$  and as ${(\F_1,\T^X_1)\in\DRevTC_{{\bf F},X}}$ we have $a_{v}^X = a_{v}^Y|_X = a_{v}^Z|_X$. Thus we deduce that  $m_v^Y\circ (m_v^Z)^{-1}$ fixes $a_{M^Z(v)}^Z|_X$ which contains at least three elements so we have $m_v^Y=m_v^Z .$
\end{proof}

According to this lemma we make an identification of $\drevTC_{{\bf F},X}$ in $\revTC_{{\bf F}}$ and we define the topology of $\drevTC_{{\bf F},X}$ as the restriction of the one in $\revTC_{{\bf F}}$. 
With this topology we are going to prove Theorem \ref{ccompact}.


We have $\rm Quad_X\subset \rm Quad_Y$. Denote by $\pi_{Y,X}$ the natural projection $$\pi_{Y,X}: \S ^{\rm Quad_Y}\to  \S ^{\rm Quad_X}.$$ In the following we define a map $\Pi_{Y,X}$ from the set of trees of spheres marked by $Y$ to the one marked by $X$. We are interested in this map because of the following observation.

\begin{lemma}\label{defPi0}
The tree $\T^X$ is compatible with $\T^Y$ if and only if $$\T^X=\Oub_{X,Y}(\T^Y).$$
\end{lemma}

\begin{proof}
Suppose that $\T^X$ is compatible with $\T^Y$. Then
each $t\in \rm Trip_{X}$ is separated by a unique vertex $v_t$ of $\Oub_{Y,X}(\T^X)$ and a unique vertex ${v'_t}$ of $\T^X$. We have $\T^X$ is compatible with $\T^Y$ and $\T^Z$ if and only if $\forall t\in Trip_{X}, a_{v'_t}^X = a_{v'_t}^Y|_X=a_{v_t}$ if and only if ${\T^X=\Oub_{X,Y}(\T^Y)}$.

Reciprocally , if $\T^X=\Oub_{X,Y}(\T^Y)$, the vertices of $\T^X$ are vertices of $\T^Y$ and by construction we have $a_v^X=a_v^Y|_X$ for all $v$ internal vertex of $T^X$.
\end{proof}

We prove in the following that this new map well behave in the quotient by the natural isomorphism relation as a map $\OubQ_{Y,X}$.

\begin{definition} We denote by $\OubQ_{Y,X}$ the map such that the following diagram commutes : 

\centerline{
$\xymatrix{\CMT_Y\ar[r]^{\Top_Y}\ar[d]_{\OubQ_{Y,X}}& \S^{ \rm Quad_Y} \ar[d]^{\pi_{Y,X}}\\
\CMT_X \ar[r]^{\Top_X}&\S^{ \rm Quad_X}.
}$} 
\end{definition}


Let $\T^Y$ be a tree of spheres marked by $Y$.
Denote by $\Pt$ the collection of partitions of $X$ associated to the vertices of $Y$ separating three elements of $X$. Recall that we defined in Definition \ref{defadmpart} admissible collections of partitions.

\begin{lemma}
The set $\Pt$ is an admissible collection of partitions.
\end{lemma}

\begin{proof}
\ref{mt1}. By definition the vertices for which we are considering the partitions separate three elements of $X$.

\ref{mt2}. Let $P$ be a partition corresponding to a vertex $v\in T^Y$ and $B\in P$. Either $B=\{x\}$, or $\card B>1$ and in this case, the branch on $V$ corresponding to $B$ contains at least an internal vertex separating two elements of $X$. Let $v'$ be one of these vertices in this branch which are the closest to $v$ (for the length of $[v,v']$). Let $e'$ be the edge on $v'$ connecting $v$ to $v'$. Then $B_{v'}(e')= (X\setminus B)$. Indeed, suppose that this is not the case, we find an element $x\in B\cap B_{v'}(e')$. Take $x_1\in B\setminus\{x\}$ and $x_2\in X\setminus B$. The vertex separating this triple $(x_1,x,x_2)$ is between $v$ and $v'$ (because $x,x_2\in B_{v'}(e')$ and $x,x_1\in B$) which contradicts the minimality of $v'$.

\ref{mt3}. Suppose by contradiction that we have $v_1$ and $v_2$ two vertices of $T^Y$ for which the associated partitions of $X$ are $P_1$ and $P_2$ and such that $P_1\cap P_2\ni B(\neq\emptyset)$. Let $B_1$(resp. $B_2$) be the branch of $v_1$ (resp. $v_2$) corresponding to $B$. As $B\in B_1\cap B_2$ we have $v_1\in B_2$ (or $v_1\in B_2$ which is a symmetric case). Let $e_1$ be the edge on $v_1$ connecting it to $v_2$. Given that $v_1$ separate three elements of $X$, we find $x\in X{\setminus} (B\cup B_{v_1}(e_1))$ which is absurd because $x\notin B_{v_1}(e_1)$ so $x\in B_2\in X=B$. 
\end{proof}

According to Corollary \ref{propsurj}, the set $\Pt$ determines a unique isomorphism class of combinatorial trees $[T^X]$. For all $t\in \rm Trip_{X}$, we denote by $v_t$ the vertex separating $t$ in $T^Y$.
Denote by $\T^X$ the tree of spheres which combinatorial tree is the representative of $[T^X]$ for which each internal vertex associated to a triple $t$ is $v_t$ and for which the map associated to each internal vertex $v$ defined by a triple $t$ is $a_v:=a_{v_t}|_{X}$. We use the notation $\Oub_{Y,X}(\T^Y):=\T^X$.

\begin{lemma}\label{defPi}
The map $\OubQ_{Y,X}$ is continuous as the quotient of the map $\Oub_{Y,X}$ by the isomorphism equivalence relation on the marked trees of spheres.
\end{lemma}

\begin{proof} Indeed, if $\T^Y_1\sim_{\cal M} \T^Y_2$ then $\Oub(\T^Y_1)\sim_{\cal M} \Oub(\T^Y_2)$. The formula follows directly from the definition of $\Oub_{Y,X}$ and as $\pi_{Y,X}$ is continuous we deduce that the map is continuous too. Moreover, $\Oub_{Y,X}$ acts on the marked spheres by restricting the marking map so it is the map previously defined.
\end{proof}

\begin{proof}(Theorem \ref{ccompact})
According to Proposition \ref{equivmod} and the defynition of its topology, the set $\revTC_{{\bf F}}$ can be identified to a subspace of ${\MT}_Y\times{\MT}_Z$.
 
According to Lemma \ref{defPi0} and Lemma \ref{defPi}, we have $$\drevTC_{{\bf F},X}=\{([\T^Y],[\T^Z])\in\revTC_{{\bf F}}\;|\; \OubQ_{Y,X}([\T^Y])=\OubQ_{Z,X}([\T^Z])\}.$$
 
So $\drevTC_{{\bf F},X}$ is a closed set in  $\revTC_{{\bf F}}$ which is compact.
\end{proof}


\subsection{Convergence}

 In $\DRevTC_{{\bf F},X}$ we define the natural convergence notion of a sequence in $\Rat_{{\bf F},X}$ as follows.
 
\begin{definition}[Dynamical convergence]\label{defcvdyn}
 A sequence $({ \F}_n,a_n^Y,a_n^Z)_n$ in $\Rat_{{\bf F},X}$ converges to $({ \F},{ \T}^X)\in \Rat_{{\bf F},X}$ if  $$\displaystyle { \F}_n\underset{\phi_n^Y,\phi_n^Z}\longrightarrow{ \F}\quad\text{with}\quad\phi_{n,v}^Y=\phi_{n,v}^Z$$ for all vertex $v$ internal vertex of $T^X$. 
\end{definition}


We prove that this dynamical convergence is compatible with the topology.

\begin{lemma}\label{ilesttemps}
 A sequence of dynamical systems between marked spheres converges to a dynamical system between trees of spheres if and only if it dynamically converges to this limit.
\end{lemma}

\begin{proof} Suppose that $(\F_n,\T^X_n)_n$ is a sequence of dynamical systems converging to a dynamical system $(\F,\T^X)$: $$\F_n\underset{(\phi^Y_n,\phi^Z_n)}\longrightarrow\F.$$

For all $t\in \rm Trip_X$, we define $\tilde\phi^Y_{n,t}=\phi^X_{n,t}$ and $\tilde\phi^Z_{n,t}=\phi^X_{n,t}$ (see notations following Remark \ref{convgeneral}). Then, for all triple ${t\in \rm Trip_Y-\rm Trip_X}$, we define $\tilde\phi^Y_{n,t}=\phi^Y_{n,t}$ and for $t\in \rm Trip_Z-\rm Trip_X$, $\tilde\phi^Z_{n,t}=\phi^Z_{n,t}$.

For all $t\in \rm Trip_X$, $(\tilde\phi^Y_{n,t})^{-1} \circ\tilde\phi^Y_{n,t}$ tends to the identity of $\S_t$ because it converges to the identity on the three elements of $t$. Thus we have $(\F_n,\T_n^X)$ converges dynamically to $(\F,\T^X)$ with respect to the families of sequences $(\tilde\phi^Y_n)_n$ and $(\tilde\phi^Z_n)_n$.
\end{proof}


We can now deduce the statement below that was assumed in \cite{A1}.
\begin{corollary}
If $(\F_n)_n$ is a sequence in $\Rat_{{\bf F},X}$, then after passing to a subsequence, there exists $(\F,\T^X)\in \DRevTC_{{\bf F},X}$ such that $(\F_n,\T^X_n)_n$ converges dynamically to $(\F,\T^X)$.
\end{corollary}

\begin{proof}This corollary follows directly from Theorem \ref{ccompact} and Lemma \ref{ilesttemps}.
\end{proof}


\begin{proposition} We have the following inclusions:
$$  \overline{\rat_{{\bf F},X}}\subsetneq\drevTC_{{\bf F},X}\subsetneq \revTC_{{\bf F}}. $$
\end{proposition}

\begin{proof} In \cite{A2} we give an example of element in $\DRevTC_{{\bf F},X}$ that is not a dynamical limit of dynamical covers between marked spheres so we have $\overline{\rat_{{\bf F},X}}\subsetneq\drevTC_{{\bf F},X}$. 
\end{proof}

\begin{remark}
In \cite{A2} are proven several general properties about dynamical systems between trees of spheres which are limits of dynamical systems between marked spheres called the annuli lemmas and branch lemma.
If a dynamical system between trees of marked spheres (in $\revTC_{\bf F,X}$) satisfies these lemmas, we can hope that it is in the closure of $\rat_{\bf F}$ but this is still an open question.
\end{remark}




\end{document}